\documentclass[]{amsart}
\usepackage{amssymb}
\usepackage{amsthm}
\usepackage{thmtools}
\usepackage{graphicx}
\usepackage{bm}
\usepackage{microtype}
\usepackage{dsfont}
\usepackage{mathtools}
\usepackage[backend=biber,bibencoding=utf-8]{biblatex}
\usepackage{csquotes}
\usepackage{comment}
\usepackage{hyperref}
\usepackage[capitalize]{cleveref}
\usepackage{enumitem}
\usepackage{color}
\usepackage{booktabs,multirow}
\usepackage{tikz}
\usepackage{subcaption}

\def\Xint#1{\mathchoice
   {\XXint\displaystyle\textstyle{#1}}%
   {\XXint\textstyle\scriptstyle{#1}}%
   {\XXint\scriptstyle\scriptscriptstyle{#1}}%
   {\XXint\scriptscriptstyle\scriptscriptstyle{#1}}%
   \!\int}
\def\XXint#1#2#3{{\setbox0=\hbox{$#1{#2#3}{\int}$}
     \vcenter{\hbox{$#2#3$}}\kern-.5\wd0}}
\def\ddashint{\Xint=}

\newcommand{\e}[1]{\ensuremath{\cdot 10^{#1}}}
\usepackage{dashbox}
\newif\ifshowboxes \showboxestrue

\DeclareMathOperator{\singsupp}{sing\,supp}
\DeclareMathOperator{\supp}{supp}

\DeclareMathOperator{\erfc}{erfc}
\DeclareMathOperator{\dist}{dist}
\DeclareMathOperator{\conv}{Conv}

\newcommand{\abs}[1]{\ensuremath{ {\lvert{#1}\rvert} }}
\newcommand{\babs}[1]{\ensuremath{ {\bigl\lvert{#1}\bigr\rvert} }}

\renewcommand{\Re}[1]{\mathrm{Re}(#1)}
\renewcommand{\Im}[1]{\mathrm{Im}(#1)}

\newcommand{\R}{\mathds{R}}
\newcommand{\C}{\mathds{C}}
\newcommand{\N}{\mathds{N}}
\newcommand{\Nz}{\mathds{N}_0}
\newcommand{\Z}{\mathds{Z}}

\newcommand{\dd}{\mathrm{d}}

\newcommand{\cL}{\mathcal{L}}

\newtheorem{theorem}{Theorem}[section]
\newtheorem{lemma}[theorem]{Lemma}
\newtheorem{corollary}[theorem]{Corollary}

\theoremstyle{definition}
\newtheorem{definition}[theorem]{Definition}

\theoremstyle{remark}
\newtheorem{remark}[theorem]{Remark}

\numberwithin{equation}{section}

\begin{document}

\title[Lattice sums without translational invariance]{On the computation of lattice sums without translational invariance}


%    author one information
\author{Andreas A. Buchheit }
\address{Department of Mathematics, Saarland University, 66123 Saarbrücken, Germany}
\curraddr{}
\email{buchheit@num.uni-sb.de}
\thanks{}

%    author two information
\author{Torsten Keßler}
\address{Department of Mechanical Engineering, Eindhoven University of Technology, 5600 MB Eindhoven, Netherlands}
\curraddr{}
\email{}
\thanks{}

\author{Kirill Serkh}
\address{Departments of Mathematics and Computer Science, University of Toronto, Toronto, ON M5S 2E4, Canada}
\curraddr{}
\email{}
\thanks{}

%    \subjclass is required.
\subjclass[2010]{Primary }

\date{}

\dedicatory{}

\begin{abstract}
This paper introduces a new method for the efficient computation of oscillatory multidimensional lattice sums in geometries with boundaries. Such sums are
ubiquitous in both pure and applied mathematics, and have immediate applications in condensed matter physics and topological quantum physics. The challenge in their evaluation results from the combination of singular long-range interactions with the loss of translational invariance caused by the boundaries, rendering standard tools ineffective. Our work shows that these lattice sums can be generated from a generalization of the Riemann zeta function to multidimensional non-periodic lattice sums. We put forth a new representation of this zeta function together with a numerical algorithm that ensures exponential convergence across an extensive range of geometries. Notably, our method's runtime is influenced only by the complexity of the considered geometries and not by the number of particles, providing the foundation for efficient simulations of macroscopic condensed matter systems. We showcase the practical utility of our method by computing interaction energies in a three-dimensional crystal structure with $3\times 10^{23}$ particles. Our method's accuracy is demonstrated through extensive numerical experiments. A reference implementation is provided online along with this article.
\end{abstract}

\maketitle

%%%%%%%%%%%%%%%%%%%%%%%%%%%%%%%%%%%%%%%%%%%%%%%%%%%%%%%%%%%%%%%%%%%%%%%%%%%%%
% Part 1: Introduction till outlook
%%%%%%%%%%%%%%%%%%%%%%%%%%%%%%%%%%%%%%%%%%%%%%%%%%%%%%%%%%%%%%%%%%%%%%%%%%%%%

\section{Introduction}

Lattice sums and their related zeta functions are of central importance in all of mathematics, with applications ranging from the study of the distribution of prime numbers to eigenvalues of pseudo-differential operators \cite{borwein2013lattice}. 
Lattice sums that involve algebraic powers of quadratic forms allow for the description of general long-range interactions in physical systems, encompassing the Coulomb interaction between charged particles, dipolar interactions in magnetic systems, and gravitation \cite{campa2014physics}. These interactions are ubiquitous in nature and give rise to numerous applications in condensed matter and quantum physics \cite{Castelnovo2008,PhysRevLett.84.5687,fey2019quantum,Richerme2014}.

The interplay of long-range interactions with boundaries has important implications for fundamental research and technological application. Topological defects in magnetic materials with dipolar long-range interactions are currently being explored as building blocks in novel spintronics devices \cite{psaroudaki2021skyrmion,song2020skyrmion,zhang2017skyrmion}, where the material boundary needs careful consideration \cite{martinez2018theory}.  Topological excitations at superconductor boundaries, called Majorana zero modes, have been long-sought as the main ingredient for error-resistant quantum computing \cite{lutchyn2018majorana}. Recent results by Microsoft Quantum strongly hint at their discovery \cite{aghaee2023inas}. Some of the authors have recently demonstrated that exotic topological phases in unconventional superconductors can arise due to long-range interactions between electrons \cite{buchheit2023exact}.

The foundation for the computation of lattice sums for complete lattices $\Lambda=A\mathds Z^{d}$, with $A\in \mathds R^{d\times d}$ regular, is given by the Epstein zeta function, the generalization of Riemann zeta to oscillatory lattice sums in higher dimensions \cite{epstein1903theorieI,epstein1903theorieII}. While efficiently computable representations of this function are well established by now \cite{crandall2012unified}, the problem of computing lattice sums with boundaries, for instance, sums over a set $L=A\Nz^d$, has remained open \cite{elizalde2012spectral_zeta}. This work solves this long-standing issue by presenting an efficiently computable framework for lattice sums and related zeta functions of the form
\[
\sum_{\bm z\in L} \frac{e^{-2\pi i \bm y\cdot \bm z}}{\vert \bm z \vert^\nu},\quad \mathrm{Re}(\nu)>d, 
\]
with $\bm y\in \mathds R^d$, including their meromorphic continuations to $\nu \in \mathds C$ for a large set of geometries $L$. We focus on recombinations of the corner $L=A\mathds Z^d-\bm x$ with $\bm x\in \mathds R^d$, including parallelepipeds and half-spaces, from which all relevant crystal structures can be constructed.

This work is designed for an interdisciplinary audience with diverse aims and backgrounds. Consequently, we have structured it as follows. In Section~\ref{sec:epstein_crandall}, we discuss the Epstein zeta function for complete lattices and present a compact reformulation of Crandall's formula that allows for its efficient evaluation. \cref{sec:main_result} presents the main results for general point sets $L$ and provides all the necessary tools for applying our method. After introducing a generalized Crandall formula for zeta functions on uniformly discrete sets, we present an efficiently computable representation for prototypical lattice subsets, including corners and parallelepipeds. In \cref{subsec:spin_circuit}, we benchmark the performance of our method and compute energies in a macroscopic 3D crystal structure of physical relevance, incorporating nontrivial boundaries. We provide the full source code alongside the article to facilitate the application of our method by readers. Our conclusions and an outlook on future applications are presented in \cref{sec:conclusions}. \cref{sec:derivation} details the derivation of the generalized Crandall formula. The detailed algorithm for numerically computing the zeta functions is laid out in \cref{sec:algorithm}. Finally, \cref{sec:experiments} benchmarks our results against analytical results and numerical experiments.  Proofs of technical lemmas are provided in the appendix.
A reference implementation of our algorithm together with a notebook that generates the figures of this article
is available at \url{https://doi.org/10.5281/zenodo.10783201}.

\section{Epstein zeta and reformulation of Crandall's formula}
\label{sec:epstein_crandall}

We begin by introducing the concept of lattices.
\begin{definition}[Lattices]
A lattice $\Lambda\subseteq \mathds R^d$ is defined as a periodic set of points of the form $
\Lambda=A \mathds Z^d$,
with $A\in \mathds R^{d\times d}$ regular. An important property of the lattice is its elementary lattice cell volume  $V_\Lambda= |\det A|$.
\end{definition}

The central mathematical object of study in the description of long-range lattice sums is the Epstein zeta function. 
\begin{definition}[Epstein zeta function]
Let $\Lambda=A\mathds Z^d$, with $A\in \mathds R^{d\times d}$ regular, $\bm x,\bm y\in \mathds R^d$, and $\nu\in \mathds C$. Then for $\mathrm{Re}(\nu)>d$, the Epstein zeta function is defined by the Dirichlet series
\begin{equation*}
  Z_{\Lambda,\nu}\left\vert \begin{matrix}
      \bm x\\\bm y
    \end{matrix}\right\vert =\,\sideset{}{'}\sum_{\bm z \in \Lambda} \frac{e^{-2\pi i  \bm y \cdot \bm z} }{ {\vert   \bm z-\bm x\vert}^{\nu}},
\end{equation*}
where the primed sum excludes the case $\bm x=\bm z$. The function is meromorphically continued to $\nu\in \mathds C$.
\end{definition}
Originally introduced at the beginning of the 20th century by Paul Epstein \cite{epstein1903theorieI,epstein1903theorieII}, it forms the natural generalization of the Riemann zeta function to oscillatory singular lattice sums in higher dimensions. Its applications span from the computation of electrostatic crystal potentials \cite{emersleben1923zetafunktionenI, emersleben1923zetafunktionenII}, over analytic number theory and statistical mechanics \cite{terras2012harmonic}, to quantum field theory \cite{elizalde2012spectral_zeta}. Two of the authors have recently developed the Singular Euler--Maclaurin expansion (SEM), a generalization of the 300-year-old Euler--Maclaurin summation formula to singular functions in higher dimensions \cite{buchheit2022efficient,buchheit2022singular} that uses the Epstein zeta function as a key element. This method has led to the prediction of two new phases in unconventional superconductors \cite{Buchheit_Continuum_representation_2022}.

Starting with early works by Chowla and Selberg \cite{chowla1949epstein}, Terras \cite{terras1973bessel}, Shanks \cite{shanks1975calculation}, and Elizalde \cite{elizalde1998multidimensional}, exponentially convergent series expansion that allow for the efficient evaluation of the Epstein zeta function in any dimension have been developed. A combination of Mellin transform and Poisson summation yields Crandall's formula~\cite{crandall2012unified}, constituting the most modern approach. Here, we compactly reformulate it as follows.
\begin{theorem}[Crandall's formula]\label{thm:crandall}
Let $\Lambda=A \mathds Z^d$ with $A\in \mathds R^{d\times d}$ regular, $\bm x,\bm y\in \mathds R^d$, and $\nu\in \mathds C\setminus\{d\}$. Then for any $\lambda>0$,
\begin{align*}
     Z_{\Lambda,\nu}\left\vert \begin{matrix}
      \bm x\\\bm y
    \end{matrix}\right\vert = 
    \begin{multlined}[t]
    \frac{(\lambda^2/\pi)^{-\nu/2}}{\Gamma(\nu/2)}\Bigg( \sum_{\bm z\in \Lambda} G_\nu( (\bm z-\bm x)/\lambda )e^{-2\pi i\bm y\cdot \bm z}  \notag \\ +\frac{\lambda^d}{V_\Lambda}\sum_{\bm k\in \Lambda^\ast} G_{d-\nu}\big(\lambda(\bm y-\bm k)\big) e^{-2\pi i \bm x \cdot (\bm y-\bm k)}\Bigg),
    \end{multlined}
\end{align*}
with the reciprocal lattice $\Lambda^\ast = A^{-T} \mathds Z^d$. The function $G_\nu$ is defined as
\begin{align*}
G_{\nu}(\bm z) = \frac{\Gamma(\nu/2,\pi \bm z ^2)}{(\pi \bm z^2)^{\nu/2}}=\int \limits_{-1}^1  \, |t|^{-\nu} e^{-\pi \bm z^2 / t^2}\,\frac{\mathrm d t}{|t|},\quad \bm z\in \mathds R^d\setminus \{\bm 0\}, 
\end{align*}
and $G_\nu(\bm 0) = -2/\nu$. Here $\Gamma(\nu,z)$ denotes the upper incomplete Gamma function.
\end{theorem}

Due to the super-exponential decay of $G_\nu$, summation both in real and in reciprocal space can be restricted to lattice elements close to $\bm 0$, allowing for an efficient and precise numerical evaluation. 

The effectiveness of Crandall's formula is based on Poisson summation, which fundamentally utilizes the periodicity of the lattice, which means that if $\bm z\in \Lambda$, then $\bm z+\Lambda=\Lambda$. A significantly more challenging task arises when the summation is restricted to a non-translationally invariant subset $L$ of a lattice, for instance  $L = \mathds{N}^d \subseteq \mathds{Z}^d$. The resulting zeta functions are of central importance in high-energy physics and frequently appear in the zeta function regularization method for path integrals developed by Hawking \cite{hawking1977zeta}. Computing these lattice sums is a well-known hard problem, for which, so far, only asymptotic expansions in the specific case \( L = \mathds{N}^2 \) have been developed \cite{elizalde2012spectral_zeta}. Only few analytic results for such lattice sums have been derived to date, most notably by Glasser and Zucker \cite{glasser1973evaluation,zucker2017exact}.  

In this work we solve this long-standing problem by providing an exponentially convergent series representation of the arising zeta functions for a vast range of lattice subsets, including corners and parallelepipeds.

\section{Main result: Computing non-translationally invariant lattice sums}
\label{sec:main_result}
The periodicity of lattices forbids the presence of boundaries. Hence, in order to describe realistic finite material structures, a more general concept is needed, namely uniformly discrete sets. 
\begin{definition}[Uniformly discrete sets]
    Let $L\subseteq \mathds R^d$. The set $L$ is said to be uniformly discrete if there exists $\varepsilon>0$ such that the Euclidean distance between any two distinct points $\bm x,\bm y\in L$ satisfies
    $\vert \bm x-\bm y\vert >\varepsilon$.
\end{definition}
In this section, we provide a general framework for computing oscillatory sums over uniformly discrete sets $L\subseteq \mathds R^d$ as described by the following zeta function.
\begin{definition}[Set zeta function]
Let $L\subseteq \mathds R^d$ be uniformly discrete, $\bm y\in \mathds R^d$, and $\nu\in \mathds C$ with $\mathrm{Re}(\nu)>d$. We then define the set zeta function 
\[
Z_{L,\nu}(\bm y)=\sideset{}{'}\sum_{\bm z\in L}\frac{e^{-2\pi i \bm y\cdot \bm z}}{\vert \bm z \vert^\nu}.
\]
\end{definition}
Note that the zeta function for the set $L$ is immediately related to the Epstein zeta function if $L$ is a shifted lattice.
\begin{remark}
  \label{rem:shift_by_x}
    Let $L\subseteq \mathds R^d$ be a lattice. Then the Epstein zeta function and the zeta function for the uniformly discrete set $L-\bm x$ with $\bm x\in \mathds R^d$ are connected as follows,
    \[
    Z_{L,\nu}\left \vert \begin{matrix} 
    \bm x \\ \bm y
    \end{matrix} \right \vert = e^{-2\pi i \bm y\cdot \bm x} Z_{L-\bm x,\nu}(\bm y).
    \]
    By the above equality, we extend the definition of the Epstein zeta function to arbitrary uniformly discrete sets $L$.
\end{remark}
Note that the set zeta function does not require an additional parameter $\bm x$, compared to the Epstein zeta function, rendering the presentation more condensed. This work will put its focus on the particularly important case that $L$ is a shifted subset of a lattice, $L+\bm x\subseteq \Lambda$, with $\bm x\in \mathds R^d$. The theory that we build, however, applies to arbitrary uniformly discrete sets.

Our work is based on the theory of tempered distributions, which we review in \cref{subsec:preliminaries}. In this context, the Fourier transform is defined as follows.
\begin{definition}[Fourier transformation]
For an integrable function $f : \mathds R^d \to \mathds C$ the Fourier transformation $\mathcal F f = \hat f$
is given by
\begin{equation*}
\hat f(\bm \xi) = \int_{\mathds R^d} f(\bm x) e^{-2 \pi i \bm x \cdot \bm \xi} \, \text d \bm x,
\quad \bm \xi \in \mathds R^d.
\end{equation*}
\end{definition}
This definition is then extended to tempered distributions by duality. Of particular importance in our treatment are discrete measures. 
\begin{definition}[Generalized Dirac comb and form factor]
Let $L\subseteq \mathds R^d$ be uniformly discrete. We define the generalized Dirac comb for $L$ as
\[
{\mathds 1}_L = \sum_{\bm z\in L} \delta_{\bm z},
\]
with $\delta_{\bm z}$ the Dirac delta distribution. The generalized Dirac Comb then defines a tempered distribution,  
whose Fourier transformation $\hat {\mathds 1}_L$ is called the form factor. Of particular importance is the set $L^\ast=\singsupp \hat {\mathds 1}_L$,
outside of which the form factor can be identified by a smooth function.
\end{definition}

Equipped with the form factor $\hat {\mathds 1}_L$, we now present the first main result of this work, the generalized Crandall formula for set zeta functions.

\begin{theorem}[Set zeta Crandall]
\label{thm:generalized-crandall}
Let $L\subseteq \mathds R^d$ be uniformly discrete and $\bm y\in \mathds R^d\setminus L^\ast$.
Then $Z_{L,\nu}(\bm y)$ can be extended to a meromorphic function in $\nu\in \mathds C$ via the representation
\begin{align*}
Z_{L,\nu}(\bm y)  = 
\frac{\pi^{\nu/2}}{\lambda^{\nu}\Gamma(\nu/2)}\bigg(\,\sideset{}{'}\sum_{\bm z\in L}
e^{-2\pi i \bm y\cdot \bm z } G_\nu\big( \bm z/\lambda\big)
+ \lambda^d \Big(\hat {\mathds 1}_L\ast G_{d-\nu}(\lambda \,\bm \cdot )\Big)(\bm y)\bigg).
\end{align*}
For fixed $\nu$, $Z_{L, \nu}$ is smooth on $\mathds R^d \setminus L^*$.
\end{theorem}

Similar to the Crandall formula in \cref{thm:crandall}, the sum is split into a contribution in real space and a contribution in Fourier space. The computation of the real space sum is simple due to the superexponential decay of $G_\nu$, which allows us to suitably truncate the sum.
The difficulty in computing the zeta function thus lies in the computation of the convolution in Fourier space. We approach this challenge by first transforming it into a one-dimensional integral involving an $\nu$-independent function.
\begin{theorem}[Hadamard integral representation]
\label{cor:hadamard_int_rep}
For $\bm y\in \mathds R^d\setminus L^\ast$, the convolution in \cref{thm:generalized-crandall} admits the analytic representation
\[
\frac{1}{\Gamma(\nu / 2)}\big(\hat {\mathds 1}_L\ast G_{d-\nu}\big)(\bm y)= \frac{2}{\Gamma(\nu / 2)} \ddashint_{0}^{1} t^{\nu-d-1} \psi_L (t/\lambda,\bm y)\,\mathrm d t,
\]
with $\psi_L\in C^\infty\big([0,1]\times (\mathds R^d\setminus L^\ast)\big)$,
\[
\psi_L(t,\bm y)=\big(\hat {\mathds 1}_{L}\ast e^{-\pi \bm \cdot^2/t^2}\big)(\bm y),\quad t>0,
\]
and smoothly extended to $t=0$.
Furthermore, the right hand side is a smooth function outside $L^*$ for any $\nu \in \mathds C$.
\end{theorem}

Here, the dashed integral denotes the meromorphic continuation of the integral to $\nu \in \mathds C$, referred to as the Hadamard integral, see~\cite{gelfand1964generalizedI}. 
The formula for $\psi_L$ can be related to a generalization of the Faddeeva function, which is explained in more detail in Section \ref{sec:generrfuneval}.

Under the condition that we can efficiently compute both $\psi_L$ and the Hadamard integral, the above formula serves as the basis for the computation of the set zeta function. Our focus now narrows to the particularly important case where $L$ is a shifted subset of a lattice $\Lambda$ with boundaries, which we refer to as lattice cuts. The prototypical geometry, from which all relevant lattice cuts can be generated, is the corner $L=A \Nz^d-\bm x$ with $\bm x\in \mathds R^d$. 
\begin{theorem}
Let $L=A \Nz^d-\bm x$, with $\bm x\in \mathds R^d$. \cref{thm:generalized-crandall} then extends to $\bm y\in L^*$, with
\begin{align*}
     L^\ast \subseteq
        \{\bm y\in \mathds R^d: A^T \bm y~\text{has at least one integer component}\},
\end{align*}
where the corresponding zeta function is meromorphic is $\nu$ with simple poles at $\nu \in (d-\N_0) \setminus (-2\N_0)$.
\end{theorem}

Computing $\psi_L$ and thus the zeta function for a non-translationally invariant set is a challenging task. We overcome this obstacle and demonstrate in this work that for $L = A \Nz^d-\bm x$, it is possible to efficiently evaluate $\psi_L$ in any number of space dimensions and thus the related zeta function. This result immediately allows for the computation of general lattice subsets, such as half-planes and parallelepipeds. This is achieved by means of a non-oscillatory integral representation of $\psi_L$.
\begin{theorem}[Integral representation of $\psi_L$]
\label{thm:integral_representation}
Consider $L=A  \Nz^{d}-\bm x $ with $\bm x\in A^{-T}  \mathds R_{< 0}^{d}$, and $\bm y\in \mathds R^d$. Then $\psi_L$ admits an efficiently computable representation, which reads for $t>0$, 
\[
\psi_L(t,\bm y)=  e^{2\pi i\langle \bm y,\bm x\rangle} e^{-\pi \vert t \bm x \vert^2}
\frac{t^{d}}{V_\Lambda}\int \limits_{\mathds R^{d}} \frac{
e^{-\pi  (A^{-T}\bm \xi)^2}}%
{\prod \limits_{j=1}^d \big(1-e^{-2\pi i t (\bm \xi - \bm Y(t))_j}\big)} \,\mathrm d  \bm \xi,
\]
where
\[
\bm Y(t)= -A^T (\bm y/t+i t \bm x).
\]
For all $\bm y\in \mathds R^d$ and $t>0$, the function $\psi_L$ can be analytically continued to $\bm x \in \mathds R^d$.
\end{theorem}

% \begin{remark}[Analytic continuation]
%     The formula for $\psi_L$ can be written as a nested sequence of integrals in $\xi_1,\dots,\xi_d$, where every integral corresponds to a Steltjes transformation. The computation of each of these Steltjes transformations, together with their analytic continuations across their respective branch cuts, is well understood [REF]. More details are given in \cref{sec:algorithm}.
% \end{remark}

A detailed description of the algorithm used to compute the zeta function for the corner is provided in Section~\ref{sec:algorithm}. Here, we discuss how to choose the splitting parameter $\lambda$, compute the Hadamard integral efficiently, and form the analytic continuation of $\psi_L$.

We can construct set zeta functions for numerous geometries from the corner through suitable recombination. The most important case is the parallelepiped.
\begin{corollary}[Parallelepiped zeta function]
    \label{cor:piped}
The zeta function for a parallelepiped-shaped sublattice  $L=A \prod_{j=1}^d \{0,1,\dots, n_j-1\}-\bm x$ with $\bm x\in \mathds R^d$ and $\bm n\in \Nz^d$ can be expressed as a sum of zeta functions for corner cuts, 
\[
Z_{L,\nu}(\bm y)= \sum_{\bm \alpha\in \{0,1\}^d} (-1)^{\bm \alpha} Z_{L_0+A c_{\bm \alpha},\nu}(\bm y), 
\]
with $L_0=A\Nz^d-\bm x$, and $(\bm c_{\bm \alpha})_j=0$ if $\alpha_j=0$ and $(\bm c_{\bm \alpha})_j = n_j$ if $\alpha_j=1$.
\end{corollary}
\begin{proof}
In view of
\[
L=A\prod_{j=1}^d \big(\Nz\setminus 
(\Nz+n_j)\big) -\bm x,
\]
the indicator function of the parallelepiped can be written as a sum of indicator functions of corners 
\[
\mathds 1_L=\sum_{\bm \alpha\in \{0,1\}^d} (-1)^{\bm \alpha} \mathds 1_{L_0+A c_{\bm \alpha} },
\]
from which the statement readily follows due to linearity of convolution and Fourier transform.
\end{proof}
Recombination of corners and parallelepipeds then yields general geometries. By introducing an integral representation of $\psi_L$ for the corner geometry in \cref{thm:integral_representation} that is non-oscillatory and thus numerically tractable, we transform our theoretical insights into an efficient numerical framework. By constructing arbitrary geometries from simple corner units through recombination, we offer a flexible approach for treating general lattice subsets. As a result, our method is universally applicable across any interaction exponent $\nu$, wave vector $\bm y$, spatial dimension $d$, and for an extensive range of geometries. We now benchmark the performance of our method in a physically relevant example. 
\section{Numerical application: Boundary effects in a macroscopic 3D spin structure}
\label{subsec:spin_circuit}

Spin systems with long-range interactions hold enormous technological potential. Defects inside these materials could be used as information carriers in spintronics devices, allowing for higher storage capacities, faster computations, and lower energy consumption compared to conventional semiconductor electronics \cite{finocchio2021promise}. In fundamental physics, spin-ice materials have been discovered where defects behave like magnetic monopoles \cite{Castelnovo2008,Bramwell2009}. The long-range dipolar interaction is highly important in their formation, leading to an effective Coulomb interaction between the monopoles \cite{Bramwell_2020}. It has recently been shown that power-law interactions between spins can be artificially created in a laboratory by placing them inside a cavity, an arrangement of highly reflective mirrors that strongly enhances laser light, where laser photons mediate an effective interaction between spins \cite{gao2021higgs}. These experiments can be used to simulate exotic phases of quantum matter such as superconductors, materials that exhibit no electric resistance and can transport electric current over arbitrary distances. Some of the authors have recently shown that long-range interactions can give rise to exotic topological phases in superconductors that have non-trivial behavior at the boundary \cite{buchheit2023exact}. However, the computation of the long-range interactions in three dimensions with an interaction exponent $\nu$ close to the system dimension, i.e., the dipolar interaction, is well-known to be notoriously difficult \cite{bramwell2001spin}.

\begin{figure}
    \centering
    \includegraphics[width=1.0\textwidth]{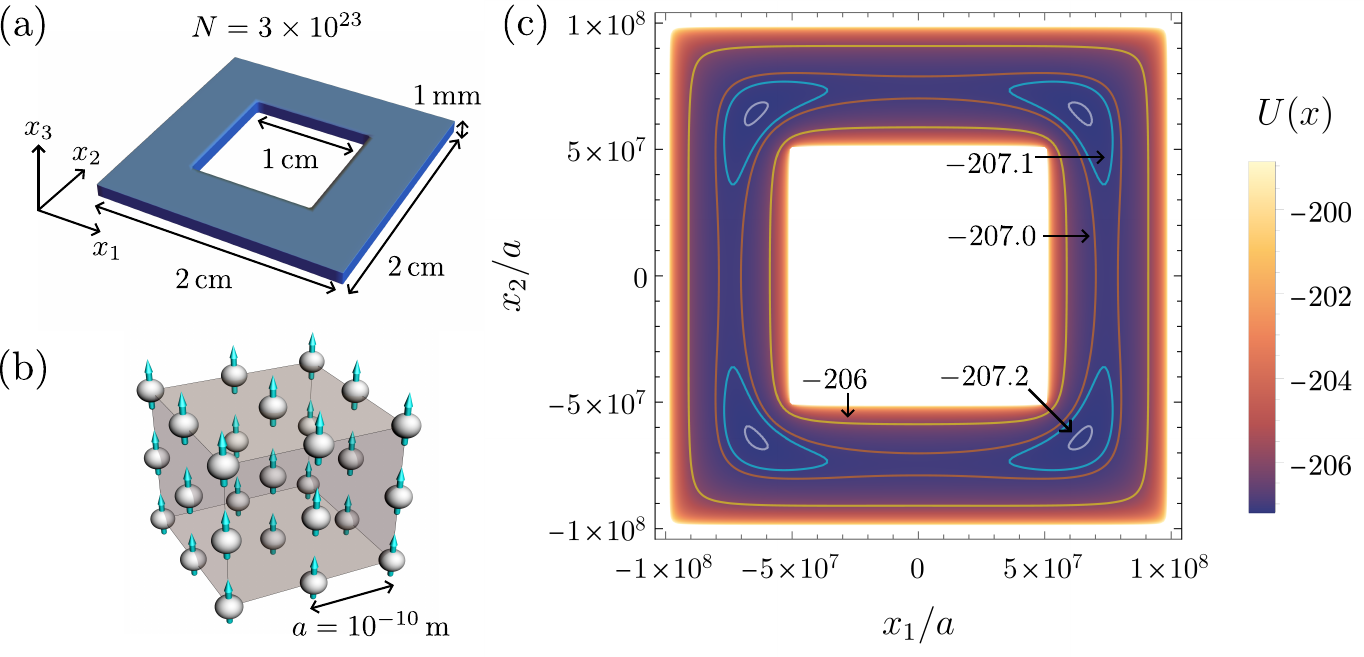}
    \caption{(a) Schematic depiction of a macroscopic three-dimensional spin circuit with inner edge length $1\,\mathrm{cm}$, outer edge length $2\,\mathrm{cm}$, and thickness $1\,\mathrm{mm}$. (b) Microscopic structure with spins (blue arrows) polarized in the $x_3$ direction. The material exhibits a cubic lattice structure and lattice constant $a=10^{-10}\,m$, amounting to $N=3\times 10^{23}$ particles in total. (b) Potential energy due to long-range interactions with exponent $\nu = 3.001$ of a spin defect obtained from inverting the spin orientation as a function of position $\bm x=(x_1,x_2,0)^T$ with $x_3 =0$ corresponding to the symmetry plane. Position is written in units of the lattice constant and potential energy in units of the interaction energy of neighboring parallel spins. Additional contours close to the energy minimum are highlighted at values $-206$ (yellow), $-207.0$ (orange), $-207.1$ (blue), and $-207.2$ (white). The figure excludes the immediate boundary layer, where the potential energy increases sharply.}
    \label{fig:spin-circuit}
\end{figure}

We now apply our method to the computation of long-range interaction energies in a macroscopic three-dimensional crystal geometry with non-trivial boundaries. Our geometry of choice is a macroscopic spin circuit, see Fig.~\ref{fig:spin-circuit} (a), defined by the domain
\[\Omega=\{\bm x \in \mathds R^3: L/4 \le \vert x_1\vert + \vert x_2\vert \le  L/2,~ \vert x_3\vert\le L/40 \},\] with $L$ the outer edge length. Microscopically, the crystal is composed of atoms arranged in a cubic lattice structure, $L=\Omega\cap\Lambda$, with $\Lambda = a \mathds Z^3$, see Fig.~\ref{fig:spin-circuit} (b). Each atom shall carry a magnetic moment or spin $\bm S(\bm x)$, described by a three-dimensional vector of unit norm (blue arrows). Choosing a lattice constant $a=10^{-10}\,\mathrm m$ and $L=2\times 10^8\,a$, our geometry has a macroscopic size with an outer edge length of $2\,\mathrm{cm}$, an inner edge length of $1\,\mathrm{cm}$, and a height of $1\,\mathrm{mm}$. In total, the structure includes more than $3\times 10^{23}$ particles. 

\begin{figure}
    \centering
    \includegraphics[width=.6\textwidth]{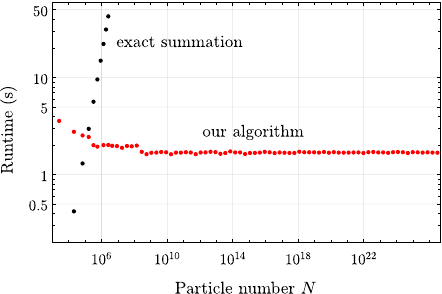}
    \caption{Runtime for evaluating the potential energy for introducing a defect in the spin circuit as a function of particle number $N$. Particle number is modified by rescaling the geometry while keeping the lattice constant fixed. While the numerical effort for exact summation increases linearly (black), our approach (red) yields an effectively constant runtime even up to macroscopic particle numbers.}
    \label{fig:runtime}
\end{figure}

The interaction energy of the spin at position $\bm x$ with all surrounding spins due to an antiferromagnetic power-law long-range interaction with exponent $\nu$ is given by
\[
U(\bm x)=\,\sideset{}{'}\sum_{\bm z\in L}\frac{S(\bm x)\cdot S(\bm z)}{\vert \bm x-\bm z\vert^\nu},
\]
with position measured in units of $a$ and energy measured in terms of interaction energy of two neighboring spins. We now assume that all spins are aligned in the $x_3$ direction and subsequently introduce a prototypical defect at position $\bm x\in L$ by inverting the corresponding spin orientation $S(\bm x)\to -\bm e_3$. The interaction energy due to long-range interactions of this defect as a function of position can then be computed from the generalized zeta function for the spin circuit. The energy as a function of position on the symmetry plane $x_3 = 0$ is displayed in \cref{fig:spin-circuit}~(c) for the particularly challenging case $\nu = 3$ (including a small positive offset of $10^{-3}$). Interactions of this kind include the isotropic part of the dipole interaction and light-induced long-range interactions in optical cavities \cite{gao2021higgs}. Additional contours close to the energy minimum are highlighted.

We find that the minimum energy is reached in four points in the material, inside the white contours. The energy increases when approaching the boundary, with sharp features appearing at the edges and corners. The immediate boundary layer, where the energy increases sharply, is excluded from the plot for better visualization. Surprisingly, for an exponent $\nu$ close to the system dimension, the potential is not flat inside the crystal, but the presence of the boundary has an effect even on macroscopic scales. This is visualized by the different contours at energies $-206$ (yellow), $-207.0$ (orange), $-207.1$ (blue), and $-207.2$ (white). This means that it is impossible to describe the material by its bulk properties only, and boundary effects always need to be included, not only for mesoscopic systems but also at macroscopic scales. Another important conclusion from our results is that long-range interactions tend to repel defects away from the boundary and localize them in the center of the material. This is highly relevant for magnetic defects as information carriers, as they can be destroyed at the boundary, which is a significant problem in technological applications \cite{martinez2018theory}. Our numerical experiment thus shows that long-range interactions can give rise to repelling boundaries that protect defects from boundary annihilation. They can even localize defects in certain regions of the geometry, in this case within the white contours.

For a macroscopic structure of this size, reference computations using exact summation are impossible, as the evaluation of a single energy evaluation requires to compute a sum with $3\times 10^{23}$ summands. In order to still be able to quantify the precision of our approach, we benchmark our results on a rescaled geometry with $L = 160\,a$, which amounts to $N\approx 1.5\times 10^5$ particles. We compare our approach with exact summation on $L$ using an equidistant grid with lattice constant $4\,a$ on the surface $x_3=0$. The maximum relative error of the energy then equals $1.2\times 10^{-12}$. Our algorithm achieves full precision, up to the condition number of the problem, for the full parameter range, as verified in \cref{sec:experiments}.

The main advantage of our method lies in a numerical effort that solely depends on the complexity of the geometry and not on the particle number. We benchmark this claim by evaluating the time required to evaluate $U(\bm x)$ at the position $\bm x = (-L/4,-L/4,0)^T$ while rescaling the geometry through rescaling the parameter $L$, making sure that $\bm x$ remains a lattice point. In Fig.~\ref{fig:runtime}, we display the single-core runtime for evaluating the energy as a function of the rescaled particle number on an Apple M1 Max processor. We observe that while the runtime for exact summation increases linearly (black points), our algorithm (red) offers a runtime that is essentially independent of the number of particles. For macroscopic particle numbers, evaluating the lattice sum amounts to a runtime of less than 2 seconds. The runtime even decreases slightly as the particle number increases, as the crystal boundaries move away from each other, allowing for a slightly faster evaluation of the arising zeta functions. 

\section{Outlook and Conclusions}
\label{sec:conclusions}

Algorithms for calculating zeta functions on lattices with boundaries have been highly sought after. However, until now, only asymptotic expansions for the 2D corner geometry $\Nz^2$ have been developed \cite{elizalde2012spectral_zeta}. The present work solves this long-standing challenge.
For any dimension, any interaction exponent, and any lattice, we present an efficiently computable representation of zeta functions for lattice subsets, focusing specifically on corners and parallelepipeds. We demonstrate that our approach enables the precise and rapid computation of energies in long-range interacting lattice systems with as many as $3\times 10^{23}$ particles and non-trivial boundaries, yielding results of significance for spintronics applications in less than two seconds on a standard laptop. While the computational cost of calculating the exact sum increases linearly with the number of particles, the computational effort for our method is constant. Its complexity depends only on the structures formed by the particles and not on their number, thereby facilitating precise and efficient simulations of macroscopic particle systems.

In future work, we plan to apply this method to analyze long-range interacting systems in condensed matter and quantum physics. We are particularly interested in boundary effects in superconductors, for example, topological excitations such as Majorana fermions, which are promising for use as error-resistant qubits. Further development of our method will explore generalizations of lattices where periodicity is no longer a given, with quasi-crystals as an initial focus to probe new physical effects in generalized point structures.

%%%%%%%%%%%%%%%%%%%%%%%%%%%%%%%%%%%%%%%%%%%%%%%%%%%%%%%%%%%%%%%%%%%%%%%%%%%%%
% Part 2: Derivation
%%%%%%%%%%%%%%%%%%%%%%%%%%%%%%%%%%%%%%%%%%%%%%%%%%%%%%%%%%%%%%%%%%%%%%%%%%%%%

\section{Derivation} \label{sec:derivation}

In this section, we provide the proofs for the results in the main section. After providing standard results in \cref{subsec:preliminaries}, we derive the generalized Crandall formula for set zeta functions in \cref{subsec:crandall_derivation}. The Hadamard integral representation for the convolution part is derived in \cref{subsec:hadamard_representation} and the efficiently computable integral representation for the corner geometry is discussed in \cref{subsec:integral_representation}.

\subsection{Preliminaries}
\label{subsec:preliminaries}
We first collect results from the theory of distributions that are needed for the
treatment of general lattice sums.
We begin with the definition of the spaces involved~\cite{horvarth2012topological}.

\begin{definition}[Test functions and distributions]
For an open set $\Omega \subseteq \mathds R^d$, we denote the space of smooth test function with compact support on $\Omega$
 by $\mathcal D(\Omega)$.
Its dual space, called the space of distributions, is written as $\mathcal D'(\Omega)$.
\end{definition}

\begin{definition}[Functions of superalgebraic decay]
For $k \in \mathds N \cup \{ \infty \}$, $\mathcal S^k(\mathds R^d)$ denotes
the space of $k$ times continuously differentiable functions of
superalgebraic decay in all derivatives up to order $k$,
\[
\sup_{\bm z \in \mathds R^d} |\bm z^{\bm \alpha} \partial^{\bm \beta} \phi(\bm z)| < \infty,
\quad \phi \in \mathcal S^k(\mathds R^d),~\bm \alpha, \bm \beta \in \mathds N^d,~|\bm \beta| \leq k.
\]
\end{definition}

\begin{definition}[Schwartz functions]
The space of smooth functions of superalgebraic decay, $\mathcal S^\infty(\mathds R^d)$, is called the Schwartz space of test functions.
In this context, we drop the exponent and simply write $\mathcal S(\mathds R^d)$.
\end{definition}

\begin{definition}[Translations]
\label{def:translat}
For $\bm y \in \mathds R^d$, let $\tau_{\bm y}$ denote the
translation by $\bm y$,
$\tau_{\bm y} \phi(\bm x) = \phi(\bm x - \bm y)$, $\bm x \in \mathds R^d$,
for a test function $\phi \in \mathcal D(\mathds R^d)$.
It is extended to distributions $u \in \mathcal D'(\mathds R^d)$ via
\[
\big(\tau_{\bm y} u\big)(\phi) = u \big( \tau_{ - \bm y} \phi \big).
\]
\end{definition}

\begin{definition}[Convolution]
For integrable functions $f, g : \mathds R^d \to \mathds C$ the convolution $f \ast g$
is given by
\[
(f \ast g)(\bm x) = \int_{\mathds R^d} f(\bm x - \bm z) g(\bm z) \, \text d \bm z
= \int_{\mathds R^d} \tau_{\bm x} \check f(\bm z) g(\bm z) \, \text d \bm z
\quad \bm x \in \mathds R^d,
\]
where $\check f$ denotes $\check f(\bm z) = f(- \bm z)$.
For $u \in \mathcal S'(\mathds R^d)$ and $\psi \in \mathcal S(\mathds R^d)$ the convolution
$u \ast \psi$ is defined as
\[
u \ast \psi(\bm x) = u\big(\tau_{\bm x} \check \psi \big), \quad \bm x \in \mathds R^d.
\]
\end{definition}

\begin{definition}[Singular support]
Let $u \in \mathcal D'(\Omega)$ for $\Omega \subseteq \mathds R^d$ open.
A point $\bm x \in \Omega$ is in the complement of the singular support of $u$, denoted by $\singsupp u$,
if there is an open neighbourhood $\Omega'$ of $\bm x$ in $\Omega$ such that the restriction of $u$ to
$\Omega'$ can be identified with a smooth function.
\end{definition}

\begin{lemma}\label{lem:fourier-schwartz}
The Fourier transformation is an automorphism on $\mathcal S(\mathds R^d)$.
By taking adjoints, $\mathcal F$ and $\mathcal F^{-1}$ extend to automorphism on $\mathcal S'(\mathds R^d)$.
\end{lemma}

In the following, we collect statements for a family of spaces that comprises the space of compactly supported test functions,
smooth functions on closed or open domains, and Schwartz functions, see~\cite{treves1967topological} for a general treatment
of Montel spaces and tensor products.

Let $\Omega \subseteq \mathds R^d$ be open or closed.
By $X(\Omega)$, we denote a space of functions defined on $\Omega$, where $X \in \{C^\infty, \mathcal D, \mathcal S \}$.

\begin{lemma}\label{lem:bootstrap-convergence}
For a bounded sequence $(\varphi_n)_n$ in $X(\Omega)$ that converges pointwise to a function $\varphi$,
\[
\lim_{n \to \infty} \varphi_n(\bm z) = \varphi(\bm z), \quad \bm z \in \Omega, 
\]
then also $\varphi \in X(\Omega)$ and the sequence converges in the topology of $X(\Omega)$ to $\varphi$.
\end{lemma}

\begin{proof}
For all $X$, $X(\Omega)$ is a Montel space, that is bounded sets are precompact.
In particular, every bounded sequence has a converging subsequence.
To prove that $(\varphi_n)_n$ converges, let $(\psi_k)_k$ denote an arbitrary subsequence.
By the Montel property, it has a converging subsequence with limit $\psi$.
Since convergence in $X(\Omega)$ implies pointwise convergence, the limit $\psi$ has to agree
with $\varphi$.
Hence, every subsequence of $(\varphi_n)_n$ has a converging subsequence with limit $\psi$.
This implies that $(\varphi_n)_n$ itsself is converging in $X(\Omega)$ to $\varphi \in X(\Omega)$.
\end{proof}

\begin{lemma}\label{lem:tensor-product}
For $\Omega_1 \subseteq \mathds R^{d_1}$, $\Omega_2 \subseteq \mathds R^{d_2}$, both open or closed,
the completion of the algebraic tensor product
\[
X(\Omega_1) \otimes X(\Omega_2)
\]
in the projective or uniform topology is given by
\[
X(\Omega_1 \times \Omega_2).
\]
\end{lemma}

\subsection{Derivation of the generalized Crandall formula}
\label{subsec:crandall_derivation}
In this section, we prove the generalized Crandall formula for zeta functions on uniformly discrete sets in \cref{thm:generalized-crandall}.

\begin{proof}[Proof of \cref{thm:generalized-crandall}]
We first restrict $\nu$ to the half plane $\Re \nu > d + k + 1$,
where $k$ is the order of the form factor $\hat{\mathds 1}_L$.
We rewrite the interaction in terms of a Mellin transform,
\[
s_\nu(\bm z) = \vert \bm z \vert^{-\nu} =  \frac{\pi^{\nu/2}}{\Gamma(\nu/2)}
\int \limits_{0}^\infty 2\, t^{\nu}e^{-\pi \bm z^2 t^2}\,\frac{\mathrm d t}{t},
\quad \bm z \in \mathds R^d \setminus \{ \bm 0 \}.
\]
Now, we split the integration interval at $t=1 / \lambda$
and use the integral representation of $G_\nu$ in order to obtain
\[
\int \limits_{1/\lambda}^\infty 2 t^{\nu}
e^{-\pi \bm z^2 t^2} \frac{\mathrm d t}{t}
= \lambda^{-\nu} G_\nu \big(\bm z/\lambda\big),
\]
which decays superexponentially as $|\bm z| \to \infty$,
and a second term,
\[
h(\bm z) = \int_0^{1/ \lambda} 2 t^{\nu - 1} e^{- \pi t^2 |\bm z|^2} \, \text dt,
\]
a smooth and bounded function with at most polynomial growth in all derivatives.
To be able to insert the Riemann splitting, we need to regularize the lattice sum
with an exponentially decaying factor,
\begin{align*}
Z_{L, \nu}(\bm y)
&= \lim_{\varepsilon \to 0} \sideset{}{'}\sum_{\bm z\in L}
e^{-\pi \varepsilon^2 |\bm z|^2} \frac{e^{-2 \pi i \bm y \cdot \bm z}}{|\bm z|^\nu} \\
&= \frac{\pi^{\nu / 2}}{\Gamma(\nu / 2)} \lim_{\varepsilon \to 0} \sideset{}{'}\sum_{\bm z\in L}
e^{-\pi \varepsilon^2 |\bm z|^2} e^{-2 \pi i \bm y \cdot \bm z} \big( \lambda^{-\nu} G_\nu(\bm z / \lambda)
+ h(\bm z) \big) \\
&= \frac{\pi^{\nu / 2}}{\Gamma(\nu / 2)} \lambda^{-\nu} \sideset{}{'}\sum_{\bm z\in L}
e^{-2 \pi i \bm y \cdot \bm z} G_\nu(\bm z / \lambda)
+ \frac{\pi^{\nu / 2}}{\Gamma(\nu / 2)} \lim_{\varepsilon \to 0} F_\varepsilon(\bm y),
\end{align*}
with
\[
F_\varepsilon(\bm y) = \sideset{}{'}\sum_{\bm z\in L} e^{- \pi \varepsilon^2 |\bm z|^2} e^{- 2 \pi i \bm y \cdot \bm z} h(\bm z).
\]
To extend the summation over the full set $L$ in above definition, we observe
\[
h(\bm 0) = \int_0^{1 / \lambda} 2 t^{\nu - 1} \, \text dt = \frac{2}{\nu} \lambda^{-\nu},
\]
so we can write
\[
F_\varepsilon(\bm y) = -\frac{2}{\nu} \lambda^{- \nu} [\bm 0 \in L]
+ \sum_{\bm z \in L} e^{- \pi \varepsilon^2 |\bm z|^2} e^{- 2 \pi i \bm y \cdot \bm z} h(\bm z),
\]
where the Iverson bracket $[\bm 0 \in L]$ equals $1$ if $L$ includes the origin and $0$ otherwise.
Thus we can write $F_\varepsilon$ as
\[
F_\varepsilon(\bm y) = -\frac{2}{\nu} \lambda^{- \nu} [\bm 0 \in L]
+ \langle e^{- \pi \varepsilon^2 |\bm \cdot |^2} e^{- 2 \pi i \bm y \bm \cdot } h, \mathds 1_L \rangle
\]
With \cref{lem:exp-int-fourier} the distributional Fourier transform of $h$ is an element of $\mathcal S^k(\mathds R^d)$,
\[
\hat h(\bm k) = \lambda^{d - \nu} G_{d - \nu}(\lambda \bm k), \quad \bm k \in \mathds R^d.
\]
Therefore,
\[
F_\varepsilon(\bm y) = -\frac{2}{\nu} \lambda^{- \nu} [\bm 0 \in L]
+ \lambda^{d - \nu} \langle \big(\varepsilon^{-d} e^{- \pi |\bm \cdot |^2 / \varepsilon^2} \ast G_{d - \nu}(\lambda \bm \cdot)\big)( \bm \cdot + \bm y),
                    \hat{\mathds 1}_L \rangle
\]
Above dual bracket in $\mathcal S(\mathds R^d)$ extends to $\mathcal S^k(\mathds R^d)$ by the assumption on $\nu$.
Furthermore, the convolution of the rescaled Gaussian and $G_{d - \nu}(\lambda \bm \cdot)$ converges in $\mathcal S^k(\mathds R^d)$
to $G_{d - \nu}(\lambda \bm \cdot)$ as $\varepsilon \to 0$ since it is the convolution of a Dirac sequence in
$\mathcal S(\mathds R^d) \subseteq \mathcal S^k(\mathds R^d)$ with an element in $\mathcal S^k(\mathds R^d)$.
Hence,
\begin{align*}
\lim_{\varepsilon \to 0} F_\varepsilon(\bm y) &=
-\frac{2}{\nu} \lambda^{- \nu} [\bm 0 \in L]
+ \lambda^{d - \nu} \langle G_{d - \nu}\big(\lambda ( \bm \cdot + \bm y)\big), \hat{\mathds 1}_L \rangle \\
&=
-\frac{2}{\nu} \lambda^{- \nu} [\bm 0 \in L]
+ \lambda^{d - \nu} \big\{\hat{\mathds 1}_L \ast G_{d - \nu}(\lambda \bm \cdot)\big\}(\bm y).
\end{align*}
To summarize,
\[
\begin{multlined}[t]
Z_{L, \nu}(\bm y) = \frac{\pi^{\nu / 2}}{\lambda^\nu \Gamma(\nu / 2)} \bigg( \sideset{}{'}\sum_{\bm z\in L}
e^{-2 \pi i \bm y \cdot \bm z} G_\nu(\bm z / \lambda)
-\frac{2}{\nu} [\bm 0 \in L] \\
+ \lambda^d \big\{\hat{\mathds 1}_L \ast G_{d - \nu}\big(\lambda \bm \cdot)\big\}(\bm y)
\bigg)
\end{multlined}
\]
With the definition of $G_\nu$ at $\bm 0$, we can include the Iverson bracket in the lattice sum and obtain
\[
Z_{L, \nu}(\bm y) = \frac{\pi^{\nu / 2}}{\lambda^\nu \Gamma(\nu / 2)} \bigg(\sum_{\bm z\in L}
e^{-2 \pi i \bm y \cdot \bm z} G_\nu(\bm z / \lambda)
+ \lambda^d \big\{\hat{\mathds 1}_L \ast G_{d - \nu}(\lambda \bm \cdot)\big\}(\bm y)
\bigg)
\]
This proves the asserted equality for $\nu \in \mathds C$ with $\Re \nu > d + k + 1$.
To extend equality to the complex plane and show smoothness in $\bm y$, we split the formula
into $S_1$, the sum over $L$, and $S_2$, the convolution in Fourier space.
Owing to the superexponential decay of $G_\nu$, $S_1$ is a smooth function of $\bm y$
for every $\nu \in \mathds C$.
Note that the potential singularity at $\nu = 0$ for $\bm 0 \in L$ is cancelled by the inverse
gamma factor in front of the sum.
The same argument shows that $S_1$ is an analytic function of $\nu$ for every fixed $\bm y \in \mathds R^d$.
With \cref{thm:integral_representation}, $S_2$ can be written as
\[
S_2 = \frac{\pi^{\nu / 2}}{\lambda^{\nu - d}} \frac{2}{\Gamma(\nu / 2)}
\ddashint_0^1 t^{\nu - d - 1} \psi_L(t / \lambda, \bm y) \, \text d t,
\]
which provides an analytic continuation of the convolution for every $\bm y \notin L^*$
and, moreover, is smooth outside $L^*$ for any $\nu \in \mathds C$.
\end{proof}

\subsection{Derivation of the Hadamard integral representation}
\label{subsec:hadamard_representation}
Having established the generalized Crandall formula in the last section,
we now prove the Hadamard integral representation for the convolution in \cref{thm:integral_representation}.

\begin{lemma}\label{lem:hadmard_representation_regular}
Let $L$ be uniformly discrete whose form factor has order $k$ and $\Re{\nu}>d+k$. Then for $\bm y\in \mathds R^d \setminus L^*$,
\[
\big(\hat {\mathds 1}_L\ast G_{d-\nu}(\lambda \bm \cdot) \big)(\bm y)=2\, \int_{0}^{1} t^{\nu-d-1} \psi_L (t/\lambda,\bm y)\,\mathrm d t,
\]
with $\psi_L(\cdot,\bm y):(0,1]\to \mathds C$ smooth and defined as
\[
\psi_L(t,\bm y)=\big(\hat {\mathds 1}_{L}\ast e^{-\pi \bm \cdot^2/t^2}\big)(\bm y), \quad t > 0.
\]
\end{lemma}

Its proof requires a few technical lemmas that we prove in \cref{sec:techincal-hadamard-representation}.

\begin{lemma}\label{lem:exp-int-sk}
The sequence $(\gamma_n)_{n \in \Nz} \subseteq \mathcal S^k(\mathds R^d)$ with
\[
\gamma_n(\bm z) = \frac{1}{n} \sum_{j = 1}^n 2 (j / n)^{\mu - 1} \exp(- \pi \lambda^2 n^2 / j^2 |\bm z|^2), \quad \bm z \in \mathds R^d,
\]
converges to $G_{-\mu}$ for every $\mu \in \mathds C$ with $\Re \mu > k + 1$, $k \in \Nz$.
\end{lemma}

\begin{lemma}
\label{lem:scaled-conv-smooth-continuous}
For $u \in \mathcal S'(\mathds R^d)$ with $U = \mathds R^d \setminus \singsupp u$, the linear operator
\[
T: \mathcal S(\mathds R^d) \to C^\infty([0, 1] \times U),~\varphi \mapsto T \varphi,
\]
with
\[
\big( T \varphi \big)(t, \bm y) = \big( u \ast \varphi(\bm \cdot / t) \big)(\bm y), \quad t > 0,~\bm y \in U,
\]
and smoothly extended to $t = 0$,
is well-defined and continuous.
If $\varphi$ is even, then $T \varphi$ has the form
\[
T \varphi(t, \bm y) = t^d \eta(t^2, \bm y), \quad t \in [0, 1],~\bm y \in U,
\]
for $\eta \in C^\infty([0, 1] \times U)$.
\end{lemma}

\begin{proof}[Proof of \cref{lem:hadmard_representation_regular}]
For $\bm y \in \mathds R^d \setminus L^*$, the smoothness of $\psi_L$ is a consequence of \cref{lem:scaled-conv-smooth-continuous}.
To prove the asserted equality, we use \cref{lem:exp-int-sk} to write
\[
G_{d - \nu}(\lambda \bm \cdot) = \lim_{n \to \infty} \frac{1}{n} \sum_{j = 1}^n 2 (j / n)^{d - \nu - 1} \exp(- \pi \lambda^2 n^2 / j^2 |\bm \cdot|^2),
\]
with convergence in $\mathcal S^k(\mathds R^d)$.
In particular,
\[
\hat{\mathds 1}_L \ast \big( G_{d  - \nu}(\lambda \bm \cdot) \big)(\bm y)
= \lim_{n \to \infty} \frac{1}{n} \sum_{j = 1}^n 2 (j / n)^{d - \nu - 1} \psi_L(j / n / \lambda, \bm y).
\]
The right hand side is a Riemann sum and converges to the corresponding integral,
\[
 \lim_{n \to \infty} \frac{1}{n} \sum_{j = 1}^n 2 (j / n)^{d - \nu - 1} \psi_L(j / n / \lambda, \bm y)
 = 2 \int_0^1 t^{\nu - d - 1} \psi_L(t / \lambda, \bm y) \, \text dt,
\]
which proves the asserted equality.
\end{proof}

We are now able to prove the main result of this subsection, the integral representation for the convolution
over the full range of $\nu \in \mathds C$

\begin{proof}[Proof of \cref{thm:integral_representation}]
\cref{lem:hadmard_representation_regular} already contains a proof of the statement for $\Re \nu > d + k + 1$.
To cover the remaining cases, we observe that owing to \cref{lem:scaled-conv-smooth-continuous}
we can write $\psi_L$ as
\[
\psi_L(t, \bm y) = t^d \eta(t^2, \bm y), \quad t \in [0, 1],~\bm y \in \mathds R^d \setminus L^*,
\]
with a smooth function $\eta$.
Hence,
\[
\lambda^d \frac{2}{\Gamma(\nu / 2)} \int_{0}^{1} t^{\nu - d - 1} \psi_L (t/\lambda,\bm y)\,\mathrm d t
=
\frac{2}{\Gamma(\nu / 2)} \int_{0}^{1} \tau^{\nu / 2 - 1} \eta(\tau/\lambda^2,\bm y)\,\mathrm d \tau,
\quad \Re \nu > 0.
\]
Now, the right hand side can be extended as an entire function by the Hadamard integral.
Note that the inverse gamma function removes the simple poles of the meromorphic Hadamard integral.
Furthermore, since $\eta$ is a smooth function in both arguments,
the Hadamard integral integral is a smooth function of $\bm y$ too.
\end{proof}

\subsection{Integral representation for corner geometry}
\label{subsec:integral_representation}
The computation of the form factor for general point sets is in general a challenging task.
For a corner $L= A \Nz^d$ in $d$ dimensions,
we can however compute the form factor and its singular support $L^*$ analytically.

\begin{lemma}
For $\mathds 1_{\Nz}$ the form factor is $\big(1 - \exp(-2 \pi i (\cdot - i 0^+))\big)^{-1}$.
\end{lemma}

\begin{proof}
For $\varepsilon > 0$ let
\[
\hat {\mathds 1}_{\varepsilon}(y) = \sum_{z \in \Nz} e^{-2 \pi i (y - i \varepsilon) z}
= \frac{1}{1 - \exp(- 2 \pi i ( y - i \varepsilon))}
, \quad y \in \mathds R.
\]
Here, the right hand side converges weakly to $\big(1 - \exp(-2 \pi i (\cdot - i 0^+))\big)^{-1}$,
see \cite[Ch.~1, Sec.~3.6.]{gelfand1964generalizedI}.
Furthermore, it is readily seen that the left hand side converges
weakly in $S'(\mathds R)$ to $\hat{\mathds 1}_{\Nz}$.
\end{proof}

From the explicit expression of the form factor we can immediately conclude its singular support.
\begin{lemma}
\label{lem:sing_supp_1_Z}
The singular support of $\hat{\mathds 1}_{\Nz}$ is $\mathds Z$.
\end{lemma}
Since the Fourier transform commutes with tensor products, we furthermore obtain:
\begin{lemma}
\label{lem:formula_1_L}
For $L =A \Nz^{d}$ the form factor reads
\[
\hat{\mathds 1}_{L} =  \bigotimes_{j = 1}^{d} \Big(1 - e^{-2 \pi i \big((A^T\,\bm \cdot)_j - i 0^+\big)}\Big)^{-1}.
\]
\end{lemma}

The explicit expression for the form factor allows a direct characterization of its singular support $L^*$.

\begin{lemma}
Let $L=A \Nz^d$ and set  $L^\ast= \singsupp \hat {\mathds 1}_L$. Then
\[
L^* \subseteq \big\{ \bm y \in \mathds R^d: A^T \bm y \text{ has at least one integer component} \big\}.
\]
\end{lemma}

\begin{proof}
    By \cref{lem:sing_supp_1_Z}, we know that $\hat {\mathds 1}_{\Nz}$ has singular support $\mathds Z$
    and hence can be identified with a smooth function $g$ on $\mathds R\setminus \mathds Z$.
    Due to the tensor representation  $\hat {\mathds 1}_{\Nz^d} = \big(\hat {\mathds 1}_{\Nz}\big)^{\otimes d}$,
    we can conclude that $\hat {\mathds 1}_{\Nz^d}$ can be identified by the smooth function $g^{\otimes d}$
    on $(\mathds R\setminus \mathds  Z)^d$. Hence we find for its singular support
    \[
    \singsupp \hat {\mathds 1}_{\Nz^d} \subseteq \mathds R^d\setminus (\mathds R\setminus \mathds  Z)^d
    =\big\{ \bm y \in \mathds R^d: \bm y \text{ has at least one integer component} \big\}.
    \]
    The assertion for the general case follows after noting that
    $\hat {\mathds 1}_{A \Nz^d}=\hat {\mathds 1}_{\Nz^d}(A_{\Lambda}^T \,\bm \cdot)$.
\end{proof}

\begin{theorem}
    Let $L = A \Nz^d-\bm x$ with $\bm x\in \mathds R^d$. Then for $t> 0$ and  $\bm y\in \mathds R^d$ 
    \[
       t\mapsto  \psi_L(t,\bm y)
    \]
    extends to a smooth function at $t=0$.
\end{theorem}

The proof of the theorem requires a series of lemmas. In the following, we use the Euler--Maclaurin expansion, see \cite{apostol1998introduction,buchheit2022singular}.
\begin{theorem}[Euler--Maclaurin expansion]\label{thm:euler-maclaurin}
Let $a,b \in \mathds Z$ with $a<b$, $\delta\in(0,1]$, $\ell\in \Nz$. For a function $f\in C^{\ell+1}[a+\delta,b+\delta]$, the Euler--Maclaurin (EM) expansion takes the form
\begin{align*}\label{eq:Euler--Maclaurin-expansion}
  \sum_{n=a+1}^b f(n)=\int\limits_{a+\delta}^{b+\delta} f(y)\,\mathrm d y &- \sum_{k=0}^ {\ell} (-1)^k\frac{B_{k+1}(1+y-\lceil y \rceil)}{(k+1)!} f^{(k)}(y)\bigg\vert^{y=b+\delta}_{y=a+\delta} \notag \\ &+ \int \limits_{a+\delta}^{b+\delta} (-1)^{\ell} \frac{B_ {\ell+1}(1+y-\lceil y \rceil)}{(\ell+1)!} f^{(\ell+1)}(y)\,\mathrm d y,
\end{align*}
with $\lceil y \rceil$ the smallest integer larger than or equal to $y$. 
Furhtermore, $B_\ell$ denote the Bernoulli polynomials, which are defined via the recurrence relation
\begin{equation*}\label{eq:bernoulli_definition}
\begin{aligned}
  B_0(y) =1,  \quad
  B'_\ell(y)=\ell B_ {\ell - 1}(y),\quad 
  \int\limits_0^1 B_\ell(y)\, \mathrm d y=0,\quad \ell\ge 1.
\end{aligned}
\end{equation*}
\end{theorem}

\begin{lemma}\label{lem:japanese-bracket-est}
For $\bm x \in \mathds R^d$, let
\[
\langle \bm x \rangle = (1 + |\bm x|^2)^{1/2}.
\]
Then for all $\bm x, \bm y \in \mathds R^d$ and $w \in \mathds R$,
\[
\langle \bm x + \bm y \rangle^w \leq 2^{|w|/2} \langle \bm x \rangle^w \langle \bm y \rangle^{|w|}
\]
\end{lemma}

\begin{lemma}
\label{lem:smooth_for_all_y}
For $u = \hat {\mathds 1}_{\Nz-x}$, $x\in \mathds R$, \cref{lem:scaled-conv-smooth-continuous} holds for all $y\in \mathds R$.
\end{lemma}
\begin{proof}
The singular support of $u = \hat{\mathds 1}_{\Nz - x}$ is
\[
\singsupp u = L^* = \mathds Z
\]
by \cref{lem:sing_supp_1_Z}.
\cref{lem:scaled-conv-smooth-continuous} proves the assertion for $y \in \mathds R\setminus L^\ast$.
Hence, we only need to study $y \in L^* = \mathds Z$.
The convolution with the form factor reads 
    \[
    \big( T \varphi \big)(t)=t \sum_{z\in \Nz-x}  e^{-2\pi i z\cdot y} \hat \varphi(t z) = e^{2\pi i z\cdot x} t \sum_{z\in \Nz}   \hat \varphi(t (z-x)),
    \quad t > 0,
    \]
    where we have used that $e^{-2\pi i y\cdot z}=1$ for $y \in \mathds Z$ and $z\in \Nz$.
    To prove smoothness of $T \varphi$ on $[0, 1]$ we show that $T \varphi \in C^{2 \ell - 1}([0, 1])$ for all $\ell \in \N$.
    Owing to the equality
    \[
    C^\infty([0, 1]) = \bigcap_{\ell \in \N} C^{2 \ell - 1}([0, 1])
    \]
    it follows $T \varphi \in C^\infty([0, 1])$.
    Furthermore, we bound all derivatives by Schwartz seminorms of $\varphi$ thus implying the continuity of the mapping $T$
    from $S(\mathds R)$ to $C^\infty([0, 1])$.
    Now, let $\ell \in \N$.
    We apply the Euler--Maclaurin expansion of order $2 \ell$ as stated in \cref{thm:euler-maclaurin},
    \[
    t\sum_{z\in \Nz}   \hat \varphi(t (z-x)) = \mathcal I(t) + \mathcal D(t) +  \mathcal R(t),
    \]
    with integral part $\mathcal I$, differential operator part $\mathcal D$, and remainder $\mathcal R$ given by
    \begin{align*}
    \mathcal I(t) &= t\int_{\mathds R_+} \hat \varphi(t(z-x)) \,\mathrm d z,\\
    \mathcal D(t) &= \sum_{k = 0}^{2\ell-1} c_k t^{k+1} \hat \varphi^{(k)}(-t x ),\\
    \mathcal R(t) &= \int_{\mathds R_+} B_{2\ell}(z) t^{2\ell +1} \hat \varphi^{(2\ell)}(t (z-x))\,\mathrm d z,
    \end{align*}
    for $t \in (0, 1]$ with constants $c_k\in \mathds R$ and $B_{2\ell}$ a continuous bounded $\mathds Z$-periodic function.
    We demonstrate smoothness of the sum by showing the smoothness of its three parts.
    For the integral part, we find after a change of variables that
    \[
        \mathcal I(t) = \int_{\mathds R_+} \hat \varphi(z-t x) \mathrm d z, \quad t \in [0, 1].
    \]
    Now,
    \[
    |\partial_t^n(\hat \varphi(z - t x))| = |x|^n |\hat \varphi^{(n)}(z - t x)|
    \leq |x|^n \| \hat \varphi \|_{n, 1} \langle z - t x \rangle^{-1}
    \leq \sqrt{2} \langle x \rangle^{n + 1} \| \hat \varphi \|_{n, 1} \langle z \rangle^{-1},
    \]
    by \cref{lem:japanese-bracket-est}, where $\| \cdot \|_{k, p}$ denotes a weighted Schwartz seminorm
    \[
    \| \psi \|_{n, p} = \sup_{z \in \mathds R} |\langle z \rangle^p \psi^{(n)}(z)|, \quad \psi \in \mathcal S(\mathds R).
    \]
    Above estimates constitute uniform integrable bounds on all derivatives with respect to $t$, valid on $[0, 1]$.
    Hence, $\mathcal I$ is a smooth function on $(0, 1]$ with finite limits of all derivatives for $t \to 0$ by the dominated
    convergence theorem.
    Consequently, $\mathcal I$ is a smooth function on $[0, 1]$ with
    \[
    |\partial_t^n \mathcal I(t)| \leq C_n \| \hat \varphi \|_{n, 1},
    \]
    for a constant $C_n > 0$ that only depends on the derivative order.
    Furthermore, $\mathcal D$ is a smooth function on $[0, 1]$ as the composition of smooth functions.
    A direct computation reveals the uniform bound in terms of seminorms of $\hat \varphi$,
    \[
    |\partial_t^n \mathcal D(t)| \leq \sum_{k = 0}^{2 \ell - 1 + n} d_k \| \hat \varphi \|_{k, 0},
    \]
    for constants $d_k > 0$.
    For the remainder $\mathcal R$ we first note that a similar argument as for $\mathcal I$
    shows that $\mathcal R$ is smooth on $(0, 1]$.
    We now prove that the first $2 \ell - 1$ derivatives of $\mathcal R$ vanish at $t=0$.
    We have that
    \begin{align*}
    \partial_t ^n \mathcal R(t) &= \sum_{m = 0}^n\int_{\mathds R_+} c_m B_{2\ell}(z) t^{2\ell +1-m} (t z)^{n-m}\hat \varphi^{(2\ell+n-m)}(t (z-x))\,\mathrm d z,\\ 
    &= \sum_{m = 0}^n c_m t^{2\ell-m} \int_{\mathds R_+}B_{2\ell}(z/t) z^{n-m}\hat \varphi^{(2\ell+n-m)}(z-tx)\,\mathrm d z,
    \end{align*} 
    for $t > 0$.
    Using that $B_{2 \ell}$ is bounded on $\mathds R$, we can establish the bound
    \[
    |\partial_t ^n \mathcal R(t)|\le
    \sum_{m=0}^n r_m t^{2\ell - m} \| \hat \varphi \|_{2 \ell + n - m, n - m + 1},
    \]
    for constants $r_m > 0$.
    It follows
    \[
    \lim_{t\to 0}  \partial_t^n \mathcal R(t) = 0,
    \]
    for $n \leq 2 \ell - 1$ and hence $\mathcal R \in C^{2 \ell - 1}([0, 1])$.

    To summarize, we have shown that all derivatives of $T \varphi$ on $[0, 1]$ are bounded by a sum of Schwartz seminorms of $\hat \varphi$.
    In particular, for all $n \in \Nz$ there exits a continuous seminorm $q_n$ on $S(\mathds R)$ with
    \[
    \sup_{t \in [0, 1]} |\partial_t^n T \varphi(t)| \leq q_n(\hat \varphi).
    \]
    Since the Fourier transform is automorphism on $S(\mathds R)$, there is a continuous seminorm $p_n$ with
    \[
    \sup_{t \in [0, 1]} |\partial_t^n T \varphi(t)| \leq q_n(\hat \varphi) \leq p_n(\varphi),
    \]
    hence proving the assertion.
\end{proof}

\begin{lemma}
Let $u_1, \dots, u_d \in S'(\mathds R)$ be tempered distributions. Fix $\bm y \in \mathds R^d$ such that the mapping $T_j$
defined as in \cref{lem:scaled-conv-smooth-continuous} for $u_j$ and $y_j$, $j = 1, \dots, d$, is well-defined and continuous.
Then, the corresponding mapping $T$ for the tensor product
\[
u = \bigotimes_{j = 1}^d u_j
\]
and $\bm y$ is well-defined and continuous.
In particular, the mapping
\[
(0, 1] \to \mathds C,~t \mapsto \Big( u \ast \varphi(\bm \cdot / t) \Big)(\bm y)
\]
extends smoothly to $t = 0$ for all Schwartz functions $\varphi \in \mathcal S(\mathds R^d)$.
\end{lemma}

\begin{proof}
We first define the related mapping 
$\tilde T: \mathcal S(\mathds R^d)\to C^\infty((0,1]^d)$ with
\[
\big( \tilde T \varphi \big)(\bm t) = \big( u \ast \varphi(\bm \cdot / \bm t) \big)(\bm y),
\quad \varphi \in \mathcal S(\mathds R^d),~\bm t \in (0, 1]^d,
\]
interpreting division by the vector $\bm t$ element-wise.
Note that if $\varphi$ is the tensor product of one-dimensional Schwartz functions,
\[
\varphi = \bigotimes_{j = 1}^d \varphi_j, \quad \varphi_j \in \mathcal S(\mathds R),
\]
the mapping $\tilde T$ factorizes,
\[
\tilde T \varphi(\bm t)
= \prod_{j = 1}^d T_j \varphi_j(t_j),
\quad \bm t \in (0, 1]^d.
\]
In particular, the right hand side if well-defined for all $\bm t \in [0, 1]^d$.
Let $\hat T$ denote the unique continuous linear continuation of
\[
\bigotimes_{j = 1}^d T_j: \bigotimes_{j = 1}^d \mathcal S(\mathds R) \to \bigotimes_{j = 1}^d C^\infty([0, 1])
\]
to the completion of the algebraic tensor products, see \cref{lem:tensor-product},
\[
\hat T: \mathcal S(\mathds R^d) \to C^\infty([0, 1]^d).
\]
By construction it agrees with $\tilde T$ on pure tensors and $\bm t \in (0, 1]^d$,
\[
\hat T \bigotimes_{j = 1}^d \varphi(\bm t) = \tilde T \bigotimes_{j = 1}^d \varphi(\bm t).
\]
Since linear combinations of pure tensors are dense, $\hat T \varphi$ provides the unique and smooth extension of $\tilde T \varphi$
to $C^\infty([0, 1]^d)$.
To recover the formulation of \cref{lem:scaled-conv-smooth-continuous}, we need to change the image space of $\tilde T$ from $C^\infty([0, 1]^d)$
to $C^\infty([0, 1])$.
This is accomplished by concatenation with the continuous map
\[
\eta : C^\infty([0, 1]^d) \to C^\infty([0, 1]),
\]
with $(\eta g)(t) = g(t, \dots, t)$, $t \in [0, 1]$.
Then $T$ is given by
\[
T = \hat T \circ \eta,
\]
by construction a continuous linear mapping from $S(\mathds R^d)$ to $C^\infty([0, 1])$ with
\[
T \varphi(t) = [u \ast \varphi(\bm \cdot / t)](\bm y),
\]
for $t \in (0, 1]$.
\end{proof}

Finally, we combine the definition of $\psi_L$ with \cref{lem:formula_1_L}
to arrive at the following proof of \cref{thm:integral_representation}.

\begin{proof}[Proof of \cref{thm:integral_representation}]
Recall that $\psi_L$ is given by the formula
\begin{align*}
\psi_L(t,\bm y)=\big(\hat {\mathds 1}_{L}\ast e^{-\pi \bm \cdot^2/t^2}\big)(\bm y).
\end{align*}
Recall also that, by \cref{lem:formula_1_L}, for $L =A \Nz^{d}$ the
form factor reads
\begin{align*}
\hat{\mathds 1}_{L} =  \bigotimes_{j = 1}^{d} \Big(1 - e^{-2 \pi i \big((A^T\,\bm \cdot)_j - i 0^+\big)}\Big)^{-1},
\end{align*}
from which it follows that
\begin{align*}
\hat{\mathds 1}_{L-\bm x} = e^{2\pi i x \cdot (\bm \cdot)}
\bigotimes_{j = 1}^{d} \Big(1 - e^{-2 \pi i \big((A^T\,\bm \cdot)_j - i 0^+\big)}\Big)^{-1}.
\end{align*}
Thus,
\begin{alignat*}{2}
\psi_{L-\bm x}(t,\bm y) &= \int_{\R^d} \frac{e^{2\pi i x \cdot (\bm y-\bm \xi)} \cdot
e^{-\pi {\bm \xi}^2/t^2}}{\prod_{j=1}^d \bigl(1-e^{-2\pi i (A^T (\bm y - \bm \xi)_j
-i0^+)}\bigr) } \, \dd \bm \xi \\
&= e^{2\pi i \bm x \cdot \bm y}  
\int_{\R^d} \frac{e^{2\pi i \bm x \cdot \bm \xi} \cdot
e^{-\pi {\bm \xi}^2/t^2}}{\prod_{j=1}^d \bigl(1-e^{-2\pi i ((A^T \bm \xi)_j + (A^T \bm y)_j
-i0^+)}\bigr) } \, \dd \bm \xi \\
&= e^{2\pi i \bm x \cdot \bm y} \cdot  e^{-\pi t^2 {\bm x}^2}
\int_{\R^d} \frac{
e^{-\pi (\bm \xi - it^2 \bm x)^2/t^2}}{\prod_{j=1}^d \bigl(1-e^{-2\pi i ((A^T \bm \xi)_j + (A^T \bm \xi)_j
-i0^+)}\bigr) } \, \dd \bm \xi \\
&= e^{2\pi i \bm x \cdot \bm y} \cdot e^{-\pi t^2 {\bm x}^2}
\int_{\R^d - it^2 \bm x} \frac{
e^{-\pi {\bm \xi}^2/t^2}}{\prod_{j=1}^d \bigl(1-e^{-2\pi i ((A^T \bm \xi)_j - t\bm Y(t)_j
-i0^+)}\bigr) } \, \dd \bm \xi,
\end{alignat*}
where $\bm Y(t) = -A^T(\bm y/t + it \bm x)$. It follows that
\begin{alignat*}{2}
\psi_{L-\bm x}(t,\bm y) &= e^{2\pi i \bm x \cdot \bm y}  e^{-\pi t^2 {\bm x}^2}
\frac{1}{V_\Lambda} \int_{A^T(\R^d - it^2 \bm x)} \frac{
e^{-\pi (A^{-T}\bm \xi)^2/t^2}}{\prod_{j=1}^d \bigl(1-e^{-2\pi i ({\xi}_j - t\bm Y(t)_j
-i0^+)}\bigr) } \, \dd \bm \xi.
\end{alignat*}
By Cauchy's theorem, assuming that $\Im{\bm Y(t)_j} \ge 0$, we have that
\begin{alignat*}{2}
\psi_{L-\bm x}(t,\bm y) &= e^{2\pi i \bm x \cdot \bm y}  e^{-\pi t^2 {\bm x}^2}
\frac{1}{V_\Lambda}
\int_{\R^d} \frac{
e^{-\pi (A^{-T}\bm \xi)^2/t^2}}{\prod_{j=1}^d \bigl(1-e^{-2\pi i ({\xi}_j - t\bm Y(t)_j-i0^+)}\bigr) } \, \dd \bm \xi,
\end{alignat*}
which proves the asserted representation of the convolution.
\end{proof}

%%%%%%%%%%%%%%%%%%%%%%%%%%%%%%%%%%%%%%%%%%%%%%%%%%%%%%%%%%%%%%%%%%%%%%%%%%%%%
% Part 3: Numerical algorithm and benchmarks
%%%%%%%%%%%%%%%%%%%%%%%%%%%%%%%%%%%%%%%%%%%%%%%%%%%%%%%%%%%%%%%%%%%%%%%%%%%%%

\section{Numerical algorithm}
\label{sec:algorithm}

Let $L=A \Nz^d$, $\bm x\in \R^d \setminus L$, $\bm y\in \R^d\setminus L^*$,
and $\nu \in \C$.  In this section, we describe a numerical method for
evaluating the function
\begin{align*}
  Z_{L,\nu}\left\vert \begin{matrix}
      \bm x\\\bm y
    \end{matrix}\right\vert =\,\sideset{}{'}\sum_{\bm z \in L}
    \frac{e^{-2\pi i  \bm y \cdot \bm z} }{ {\vert   \bm z-\bm
    x\vert}^{\nu}},
\end{align*}
which is directly related to the set zeta function via \cref{rem:shift_by_x}.
We first observe that, by dividing the corner $L=A\Nz^d$ into subsets along the
planes $\{ \bm z \in \R^d : (A^{-1}\bm z)_j = (A^{-1}\bm x)_j \}$, $j=1,2,\ldots,d$,
we can represent 
this zeta function for $\bm x\in \R^d \setminus L$ by a 
combination of corner zeta functions, each corresponding to the case
when $\bm x\in A \mathds R_{\le
0}^{d}\setminus L$.
We then use \cref{rem:shift_by_x} together with \cref{thm:generalized-crandall},
\cref{cor:hadamard_int_rep}, and \cref{thm:integral_representation}
to perform our calculation.  Recall that the
sum over $\bm z$ appearing
in the formula for $Z_{L-\bm x,\nu}(\bm y)$ involves the summands $G_\nu(\bm
z/\lambda)$, which decay exponentially in~${\bm z}^2$. These summands
can be written in terms of the upper incomplete gamma function, for which
efficient and accurate numerical codes are available (see, for
example,~\cite{gil2012efficient}). The main difficulty is thus the
evaluation of the integral $I_{L-\bm x}(\bm y) = \ddashint_0^1 t^{\nu-d-1}
\psi_{L-\bm x}(t/\lambda,\bm y)\, \dd t$ and of the function $\psi_{L-\bm
x}$.  The numerical evaluation of this zeta function using these formulas
depends on resolution of the following three problems:
\begin{enumerate}

\item The choice of the Riemann splitting parameter~$\lambda$;

\item The evaluation of the function $\psi_{L-\bm x}$;

\item The evaluation of the singular Hadamard finite-part integral $I_{L-\bm
x}(\bm y)$.

\end{enumerate}

\subsection{Choice of the Riemann splitting parameter~$\lambda$}
\label{sec:splitting}

The Riemann splitting parameter $\lambda$ controls how much of the set zeta
function is computed in the sum appearing in the representation given by
\cref{thm:generalized-crandall}, and how much of it is computed in the
integral $I_{L-\bm x}(\bm y)$.  A smaller value of $\lambda$ means a smaller
number of terms in the sum, as well as a larger domain of integration in the
integral $I_{L-\bm x}(\bm y)$, and vice versa.

First, recall that
\begin{align*}
G_\nu(\bm z) = \frac{\Gamma(\nu/2,\pi \bm z^2)}{(\pi \bm z^2)^{\nu/2}} \sim
\frac{e^{-\pi \bm z^2}}{\pi \bm z^2} \qquad \text{as $\vert \bm z \vert \to \infty$}.
\end{align*}
Thus, for each $\varepsilon > 0$, $G_\nu(\bm z/\lambda) \ge \varepsilon$
whenever
\begin{align*}
\vert \bm z\vert \lesssim \lambda \sqrt{\frac{\log(1/\varepsilon)}{\pi}}.
\end{align*}
Letting $\sigma_1 \ge \sigma_2 \ge \cdots \ge \sigma_d > 0$ denote the eigenvalues
of $A^T A$, we see that the sum over $\bm z\in A[-m,m]^d \cap (L-\bm x)$,
where 
\begin{align*}
m = \frac{\lambda}{\sqrt{\sigma_d}} \sqrt{\frac{\log(1/\varepsilon)}{\pi}},
\end{align*}
contains all $\bm z\in L-\bm x$ for which $G_\nu(\bm z/\lambda) \gtrsim
\varepsilon$.

In order to simplify the evaluation of the function $\psi_{L-\bm x}$,
we require the support of the Gaussian
$e^{-\pi (A^{-T} \bm \xi)^2}$ to contain, to precision $\varepsilon$, at
most a single element of the set $\prod_{j=1}^d(\Re{\bm Y(t)_j}+\Z/t)$. Since
$\Re{\bm Y(t)_j}$ is arbitrary, this means we require that that $e^{-\pi
(A^{-T} \bm \xi)^2} \le \varepsilon$ for all $\bm \xi \in \R^d \setminus
[-h/(2t),h/(2t)]^d$, for some $0 < h < 1$. 
It is easy to see that this holds whenever
\begin{align*}
\lambda \ge \frac{2\sqrt{\sigma_1}}{h} \sqrt{\frac{\log(1/\varepsilon)}{\pi}}.
\end{align*}
In order to ensure that the support of the Gaussian is well-separated from
all other elements of $\prod_{j=1}^d(\Re{\bm Y(t)_j}+\Z/t)$ besides the
single element inside the cell $[-1/(2t),1/(2t))^d$, we choose $h=1/2$. 
Note that, if we set $\lambda$ equal to this lower bound, then the total
number of elements in the set $A[-m,m]^d$ grows as
$O((\sigma_1/\sigma_d)^{d/2}) = O(\kappa(A)^d)$, where $\kappa(A)$ is the
condition number of $A$.

In order to minimize the cost of evaluating the sum appearing in the
representation of the zeta function, we would like to make $\lambda$ as
small as possible, which means setting it equal to this lower bound.  There
is, however, a situation in which $\lambda$ can be made larger, without
incurring any extra cost in the evaluation of the sum.
Recall that $\bm x \in A \R_{\le 0}^{d} \setminus L$, and
suppose that 
\begin{align*}
\dist(\bm x,\conv(L)) \ge \lambda \sqrt{\frac{\log(1/\varepsilon)}{\pi}},
\end{align*}
where $\conv(L) \subseteq \R^d$ denotes the convex hull of the set $L$.
In this case, there will be no terms of size larger than $\varepsilon$
appearing in the sum.  Equivalently, the sum will contain no terms larger
than $\varepsilon$, whenever
\begin{align*}
\lambda \le \dist(\bm x,\conv(L)) \sqrt{ \frac{\pi}{\log(1/\varepsilon)} }.
\end{align*}
In order to avoid computing the distance between $\bm x$ and $\conv(L)$, we
use the fact that $\sigma_n (A^{-1}\bm x)^2 \le \dist(\bm x,\conv(L))^2$ to
arrive at the looser but simpler bound
\begin{align*}
\lambda \le \sqrt{ \frac{\pi \sigma_n}{\log(1/\varepsilon)}\cdot (A^{-1}\bm x)^2}
\le \dist(\bm x,\conv(L)) \sqrt{ \frac{\pi}{\log(1/\varepsilon)} }.
\end{align*}
Since we would like to make $\lambda$ as small as possible in order to
reduce the number of terms in the sum, and since we would like $\lambda$ to
be large when the sum is empty in order to reduce the domain of $I_{L-\bm
x}(\bm y)$, we set
\begin{align*}
\lambda = 
\max\Biggl(
\frac{2}{h} \sqrt{\frac{\sigma_1 \log(1/\varepsilon)}{\pi}},
\sqrt{ \frac{\pi \sigma_d}{\log(1/\varepsilon)}\cdot (A^{-1}\bm x)^2}\Biggr).
    \label{eq:lambdafor}
\end{align*}
%

% \subsection{Evaluation of the function $\psi_{L-\bm x}$}
% \subsection{Evaluation of the function}
\subsection{Evaluation of the function
\texorpdfstring{$\psi_{L-\bm{x}}$}{$\psi_{xL-x}$}}
\label{sec:generrfuneval}

The representation for $\psi_{L-\bm x}$ 
provided  in \cref{thm:integral_representation} can be
understood in terms of a generalization of the classical Faddeeva
$w$ function
\begin{align*}
w(z) = \frac{1}{\pi i} \int_\R \frac{e^{-t^2}}{t-z} \, dt, \qquad \text{$\Im{z} > 0$},
\end{align*}
where the definition is extended to all $z\in \C$ by analytic continuation.
Note that $\erfc(x) = e^{-x^2} w(ix)$, where
\begin{align*}
\erfc(x) = 1-\frac{2}{\sqrt{\pi}} \int_0^x e^{-t^2}\, dt, \qquad \text{$x \in \C$}.
\end{align*}

\begin{definition}[Generalized Faddeeva function]
  \label{def:genfadd}
Suppose that $u \in \mathcal D'(\R^d)$ and that $G \in \R^{d \times d}$ is a positive
definite matrix. The generalized multidimensional Faddeeva function $w
\colon \C^d \to \C$ is defined by the formula
\begin{align*}
w(\bm z; u, G) = \langle \tau_{\bm z} u(\bm \xi), e^{-\bm \xi^T G^{-1} \bm \xi} 
\rangle, \qquad \text{$\bm z \in \C^d$},
\end{align*}
with the translation operator $\tau_{\bm z}$ from \cref{def:translat}.

\end{definition} 

Note that $w(z; \frac{1}{\pi i (\cdot-i0^+)}, 1) = w(z)$. The following theorem
expresses the function $\psi_{L-\bm x}$ in terms of the generalized Faddeeva
function, and follows immediately from \cref{thm:integral_representation}.

\begin{theorem}
\label{thm:psi_Lmx_faddeeva}
Let $L=A\Nz^d$ and suppose that $\bm x \in \R^d$ and $\bm y\in \R^d$. Then, for
all $t>0$, 
\begin{align*}
\psi_{L-\bm x}(t,\bm y) = 
e^{2\pi i \bm x \cdot \bm y}  e^{-\pi t^2 {\bm x}^2}
\frac{1}{V_\Lambda}
w(\bm Y(t); \otimes_{j=1}^d t(1-e^{-2\pi it(\cdot-i0^+)})^{-1}, A^T A/\pi),
\end{align*}
where
\begin{align*}
\bm Y(t)= -A^T (\bm y/t+i t \bm x).
\end{align*}

\end{theorem}

In the remainder of this section, we describe a numerical algorithm for the
accurate and efficient evaluation of the generalized Faddeeva function
\begin{align*}
w(\bm z; \otimes_{j=1}^d f({\cdot-i0^+}), G),
\end{align*}
where $\bm z\in \C^d$, $G\in \R^{d\times d}$ is a positive definite matrix,
and $f\colon \C\setminus (\Z/t) \to \C$ is a meromorphic function with poles
at the elements of $\Z/t$, each with a residue of $(2\pi i)^{-1}$, satisfying
the identity $f(z\pm 1/t) = f(z)$.  We assume
that the support of the Gaussian $e^{-\bm \xi^T G^{-1} \bm \xi}$ is, to
precision $\varepsilon$, contained in the set $[-1/(4t),1/(4t)]^d$. 

We proceed by forming a Cholesky decomposition $G=\cL\cL^T$ of the matrix $G$,
where $\cL$ is the lower triangular Cholesky factor.  Such a decomposition
always exists, since the matrix $G$ is positive definite.  Let $U=\cL^{-T}$,
so that $G^{-1}=\cL^{-T}\cL^{-1} = UU^T$. We then define $\bm\eta(\bm \xi) =
U^T \bm\xi$, and observe that $\eta_{j}(\bm \xi)$ only depends on $\xi_{1},
\xi_{2}, \ldots, \xi_{j}$. This allows us to write down the following
recursive formula.  Let
\begin{align*}
\phi(\bm z) = w(\bm z; \otimes_{j=1}^d f({\cdot -i0^+}), G),
\end{align*}
for $\bm z \in \C^d$, and let the functions $\phi^{(k)}\colon \C^d \to \C$,
$0\le k\le d-1$, be given by the formula
\begin{align*}
\phi^{(k)}(\bm z^{(k)}) = \int_\R e^{-\eta_{k+1}(\bm z^{(k+1)})^2}
f(\xi_{k+1} - z_{k+1} - i0^+) \phi^{(k+1)}(\bm z^{(k+1)}) \, \dd \xi_{k+1},
\end{align*}
when $\Im{z_{k+1}} \ge 0$, and let
\begin{align*}
\begin{split}
&\phi^{(k)}(\bm z^{(k)}) = 
\sum_{n=-\infty}^\infty e^{-\eta_{k+1}(\bm z^{(k)}+(n/t){\bm e}_{k+1})^2}
\phi^{(k+1)}(\bm z^{(k)}+(n/t){\bm e}_{k+1}) \\ 
&\hspace*{4.5em} +
\int_\R e^{-\eta_{k+1}(\bm z^{(k+1)})^2}
f(\xi_{k+1} - z_{k+1} -i0^-) \phi^{(k+1)}(\bm z^{(k+1)}) \, \dd \xi_{k+1},
\end{split}
\end{align*}
when $\Im{z_{k+1}} < 0$,
where $\bm z^{(k)} = (\xi_1,\xi_2,\ldots,\xi_k,z_{k+1},z_{k+2},\ldots,z_d)$, 
$\bm z\in \C^d$, $\bm e_{k+1}$ is the $(k+1)$th unit vector, and
$\phi^{(d)}(\bm \xi) = 1$.  We then have that $\phi^{(0)}(\bm z) = \phi(\bm
z)$.

Since $\phi(\bm z)$ can become large when $\Im{z_j} < 0$ for some $1\le j\le
d$, we instead compute a scaled version $\tilde \phi(\bm z) =
e^{-\beta} \phi(\bm z)$, where $\beta$ is defined as follows. Let $\bm{\bar
z}$ be given by the formula $\bar z_j = \min(\Im{z_j},0)$ for $1\le j\le
d$.  We then define $\beta = \sum_{j=1}^n \eta_j(\bm{\bar z})^2$.  It is
possible to show from \cref{def:genfadd} that $\tilde\phi(\bm z) \lesssim
1$. We now observe that that $\tilde \phi(\bm z)$ satisfies the following
recursive formula.  Let the functions $\tilde \phi^{(k)}\colon \C^d \to \C$, $0\le k\le d-1$, be given by the formula
\begin{align*}
\begin{split}
&{\tilde\phi}^{(k)}(\bm z^{(k)}) = \int_\R e^{-\eta_{k+1}(\bm
z^{(k+1)})^2-\eta_{k+1}(\bm{\bar z})^2} \\
&\hspace*{9.5em}
f(\xi_{k+1} - z_{k+1} - i0^+) {\tilde\phi}^{(k+1)}(\bm z^{(k+1)}) \, \dd \xi_{k+1},
\end{split}
\end{align*}
when $\Im{z_{k+1}} \ge 0$, and let
\begin{align*}
\begin{split}
&{\tilde\phi}^{(k)}(\bm z^{(k)}) = 
\sum_{n=-\infty}^\infty e^{-\eta_{k+1}(\bm z^{(k)}+(n/t){\bm e}_{k+1})^2
-\eta_{k+1}(\bm{\bar z})^2}
{\tilde\phi}^{(k+1)}(\bm z^{(k)}+(n/t)\bm{e}_{k+1}) \\ 
&\hspace*{4.5em} +
\int_\R e^{-\eta_{k+1}(\bm z^{(k+1)})^2-\eta_{k+1}(\bm{\bar z})^2} \\
&\hspace*{9.5em} 
f(\xi_{k+1} - z_{k+1} -i0^-) {\tilde\phi}^{(k+1)}(\bm z^{(k+1)}) \, \dd \xi_{k+1},
\end{split}
\end{align*}
when $\Im{z_{k+1}} < 0$, where $\bm z^{(k)} =
(\xi_1,\xi_2,\ldots,\xi_k,z_{k+1},z_{k+2},\ldots,z_d)$, $\bm z \in \C^d$, 
$\bm e_{k+1}$ is the
$(k+1)$th unit vector, and $\tilde\phi^{(d)}(\bm \xi) = 1$.  It then
holds that $\tilde\phi^{(0)}(\bm z) = \tilde\phi(\bm z)$. 

This recursive formula can be further simplified to eliminate the
infinite sum. Let $\bm{\tilde z}$ denote the single point in the set $(\bm
z+\Z^d/t) \cap \{ z\in \C : \Re{tz} \in [-1/2,1/2)\}^d$. Clearly,
$\phi(\bm{\tilde z}) = \phi(\bm z)$.  Consider now the infinite sum
\begin{align*}
\sum_{n=-\infty}^\infty e^{-\eta_{k+1}(\bm z^{(k)}+(n/t){\bm e}_{k+1})^2
-\eta_{k+1}(\bm{\bar z})^2}
\end{align*}
which appears in this recursive formula in the case $\Im{z_{k+1}} < 0$.  If
we assume that $\eta_{k+1}(\Im{{\bm z}^{(k)}})^2 \lesssim
\eta_{k+1}(\bm{\bar z})^2$, for $0\le k\le d-1$, then we observe that
\begin{align*}
\babs{e^{-\eta_{k+1}(\bm z^{(k)}+(n/t){\bm e}_{k+1})^2
-\eta_{k+1}(\bm{\bar z})^2}}
\lesssim \babs{e^{-\eta_{k+1}(\Re{\bm z^{(k)}}+(n/t){\bm
e}_{k+1})^2 }},
\end{align*}
for $n\in \Z$. Recalling that the support of $e^{-\bm \xi^T G^{-1} \bm \xi}$
is contained in the set $[-1/(4t),1/(4t)]^d$, to precision $\varepsilon$, we
see that
\begin{align*}
\babs{e^{-\sum_{j=1}^{k+1} \eta_{j}(\Re{\bm{\tilde z}^{(k)}}+(n/t){\bm
e}_{k+1})^2 }} \lesssim \varepsilon
\end{align*}
for all $n\in \Z\setminus \{0\}$,
where $\bm{\tilde z}^{(k)} = (\xi_1,\xi_2,\ldots,\xi_k,\tilde z_{k+1},
\tilde z_{k+2},\ldots,\tilde z_d)$. This means that the final contribution
of all terms in the infinite sum, besides the one corresponding to
$\bm{\tilde z}^{(k)}$, to the value of $\tilde\phi^{(0)}(\bm{z})$, is
less than $\varepsilon$ in size. Since our goal is to compute
$\tilde\phi^{(0)}(\bm{z})$ to precision $\varepsilon$, this means that
these terms can be dropped from the sum. In other words,
\begin{align*}
\sum_{n=-\infty}^\infty e^{-\eta_{k+1}(\bm z^{(k)}+(n/t){\bm e}_{k+1})^2
-\eta_{k+1}(\bm{\bar z})^2} \approx
e^{-\eta_{k+1}(\bm{\tilde z}^{(k)})^2
-\eta_{k+1}(\bm{\bar z})^2},
\end{align*}
in the sense that the neglected terms have a contribution to the 
value $\tilde\phi^{(0)}(\bm{z})$ of less than $\varepsilon$. 
If we
now let the functions $\tilde
\phi^{(k)}\colon \C^d\to \C$, $0\le k\le d-1$, be given by
the formula
\begin{align*}
\begin{split}
&{\tilde\phi}^{(k)}(\bm{\tilde z}^{(k)}) = \int_\R e^{-\eta_{k+1}(\bm{
\tilde z}^{(k+1)})^2-\eta_{k+1}(\bm{\bar z})^2} \\
&\hspace*{9.5em}
f(\xi_{k+1} - \tilde z_{k+1} - i0^+) {\tilde\phi}^{(k+1)}(\bm z^{(k+1)}) \, \dd \xi_{k+1},
\end{split}
\end{align*}
when $\Im{\tilde z_{k+1}} \ge 0$, and let
\begin{align*}
\begin{split}
&{\tilde\phi}^{(k)}(\bm{\tilde z}^{(k)}) = 
e^{-\eta_{k+1}(\bm{\tilde z}^{(k)})^2
-\eta_{k+1}(\bm{\bar z})^2}
{\tilde\phi}^{(k+1)}(\bm{\tilde z}^{(k)}) \\ 
&\hspace*{4.5em} +
\int_\R e^{-\eta_{k+1}(\bm{\tilde z}^{(k+1)})^2-\eta_{k+1}(\bm{\bar z})^2} \\
&\hspace*{9.5em} 
f(\xi_{k+1} - \tilde z_{k+1} -i0^-) {\tilde\phi}^{(k+1)}(\bm{\tilde z}^{(k+1)}) \,
\dd \xi_{k+1},
\end{split}
\end{align*}
when $\Im{\tilde z_{k+1}} < 0$, 
where $\bm{\tilde z}^{(k)} = (\xi_1,\xi_2,\ldots,\xi_k,\tilde z_{k+1},
\tilde z_{k+2},\ldots,\tilde z_d)$, $\bm{z}\in \C^d$,
and $\tilde\phi^{(d)}(\bm \xi) = 1$, then we have that
$\tilde\phi^{(0)}(\bm{\tilde z}) \approx \tilde\phi(\bm z)$ to precision
$\varepsilon$.

One remaining issue with this formula is that, if any of the terms
$\xi_1,\xi_2,\ldots,\xi_k$ appearing in the vector $\bm{\tilde z}^{(k)} =
(\xi_1,\xi_2,\ldots,\xi_k,\tilde z_{k+1}, \tilde z_{k+2},\ldots,\tilde z_d)$
have a large imaginary part, then the exponential term in the integrand will
be highly oscillatory.  We would like to perform a change of variables on
$\xi_{k+1}, \xi_{k+2},\ldots, \xi_d$ which eliminates the imaginary parts of
$\xi_1,\xi_2,\ldots,\xi_k$, while keeping the values of $\eta_j(\bm{\xi})$
unchanged for $j=k+1,k+2,\ldots,d$.
First, we observe that $UU^T=G^{-1}$, so $U^T=U^{-1}G^{-1}$. Recall
that $U^{-1}$ is upper triangular. If we let $g_1,g_2,\ldots,g_d$ denote the
columns of $G$, then it is easy to see that $(U^T g_\ell)_j =
(U^{-1}G^{-1}g_\ell)_j = 0$ for all $1\le \ell< j\le d$. This means that
$\eta_j(\bm{\xi} + \alpha g_\ell) =\eta_j(\bm{\xi})$ for all $1\le \ell
<j\le d$ and $\alpha \in \R$. Let
\begin{align*}
\Delta_k(\bm\xi) = -i(g_1|g_2|\cdots|g_k) G_k^{-1}
(\Im{\xi_1},\Im{\xi_2},\ldots,\Im{\xi_k})^T,
\end{align*}
where $G_k$ is the leading principal minor of $G$ of order $k$. Since $G$ is
positive definite, Sylvester's criterion tells us $G_k$ has a positive determinant
for each $1\le k\le d$, so $\Delta_k(\bm\xi)$ is well-defined.
Then, $\eta_j(\bm\xi + \Delta_k(\bm{\xi})) = 
\eta_j(\bm{\xi})$ for $j=k+1,k+2,\ldots,d$ and $\Im{(\bm\xi +
\Delta_k(\bm{\xi}))_j} = 0$ for $j=1,2,\ldots,k$.
Making the change of variables $\xi_j'=\xi_j+\Delta_k(\bm{\tilde
z}^{(k)})_j$ for $j=k+1,k+2,\ldots,d$, we have that
$\tilde\phi^{(k)}(\bm{z} + \Delta_k(\bm{z}))
=\tilde\phi^{(k)}(\bm{z})$, for any $\bm z\in\C^d$. 

This leads us to the final form of the recursive formula.
Let the functions $\tilde\phi^{(k)}\colon \C^d \to\C$, $0\le k\le d-1$,
be given by the formula
\begin{align*}
\begin{split}
&{\tilde\phi}^{(k)}(\bm{\tilde z}^{(k)}) = \int_\R e^{-\eta_{k+1}(\bm{
\tilde z}^{(k+1)})^2-\eta_{k+1}(\bm{\bar z})^2} \\
&\hspace*{9.5em}
f(\xi_{k+1} - \tilde z_{k+1} - i0^+) {\tilde\phi}^{(k+1)}(\bm{\tilde z}^{(k+1)}) \,
\dd \xi_{k+1},
\end{split}
\end{align*}
when $\Im{\tilde z_{k+1}} \ge 0$, and let
\begin{align*}
\begin{split}
&{\tilde\phi}^{(k)}(\bm{\tilde z}^{(k)}) = 
e^{-\eta_{k+1}(\bm{\tilde z}^{(k)})^2
-\eta_{k+1}(\bm{\bar z})^2}
{\tilde\phi}^{(k+1)}(\bm{\tilde z}^{(k)} + \Delta_{k+1}(\bm{\tilde z}^{(k)})) \\ 
&\hspace*{4.5em} +
\int_\R e^{-\eta_{k+1}(\bm{\tilde z}^{(k+1)})^2-\eta_{k+1}(\bm{\bar z})^2} \\
&\hspace*{9.5em} 
f(\xi_{k+1} - \tilde z_{k+1} -i0^-) {\tilde\phi}^{(k+1)}(\bm{\tilde z}^{(k+1)}) \,
\dd \xi_{k+1},
\end{split}
\end{align*}
when $\Im{z_{k+1}} < 0$, where $\bm{\tilde z}^{(k)} =
(\xi_1,\xi_2,\ldots,\xi_k,\tilde z_{k+1}, \tilde z_{k+2},\ldots,\tilde
z_d)$, $\bm{z}\in \C^d$, and $\tilde\phi^{(d)}(\bm \xi) = 1$.  Then
$\tilde\phi^{(0)}(\bm{\tilde z}) \approx \tilde\phi(\bm z)$ to precision
$\varepsilon$.

In order to evaluate this recursive formula, it is necessary for us to
compute the integrals
\begin{align*}
I_j = \int_\R e^{-\eta_{j}(\bm{\tilde z}^{(j)})^2
-\eta_{j}(\bm{\bar z})^2} f(\xi_{j} - \tilde z_{j} - i0^\pm)
\tilde \phi^{(j)}(\bm{\tilde z}^{(j)}) \, \dd
\xi_{j},
\end{align*}
where $f(z)=t(1-e^{-2\pi itz})^{-1}$,
for $1 \le j\le d$. Recall that $\bm\eta(\bm \xi)=U^T\bm \xi$, where 
$U$ is an upper triangular matrix such that
$G^{-1}=UU^T$.
To evaluate $I_j$, we first let $\alpha_j \in \R$ and $\beta_j \in
\C$ be defined by
\begin{align*}
\alpha_j = U_{jj} \qquad \text{and} \qquad \beta_j = \sum_{k=1}^{j-1} U_{kj}
\xi_k,
\end{align*}
so that $\eta_j(\bm{\tilde z}^{(j)}) = \alpha_j \xi_j + \beta_j$, and
let $w_j = \eta_j(\bm{\bar z}^{(j)})$. We evaluate
$I_j$ to a precision of $\varepsilon$. Letting $a_j = (-\beta_j -
\sqrt{\log(1/\varepsilon)})/\abs{\alpha_j}$ and $b_j = (-\beta_j +
\sqrt{\log(1/\varepsilon)})/\abs{\alpha_j}$, we thus truncate
the domain of integration from $\R$ to $[a_j,b_j]$. When $\tilde z_j$ is
well-separated from $\R$, the function $f(\xi_j - \tilde z_j)$ is 
smooth. Since we know that
$\tilde \phi^{(j)}(\xi_{k+1},\xi_{k+2},\ldots,\xi_{j-1},\xi_j,\allowbreak \tilde
z_{j+1},\ldots, \tilde z_d)$ is also a smooth
function of $\xi_j$, we can evaluate the integral using a standard adaptive
Gauss-Legendre quadrature scheme, like the one described in~\cite{chen2022adaptive}. 

Note however that, while $\tilde\phi^{(j)}$ is smooth, it has, relative
to the scale of $[a_j,b_j]$, features on the scale of $\vert U_{jj}/U_{jk}
\vert$ for $k=j+1,j+2,\ldots,d$, so that the finest feature is on the scale
$\min_{k>j} \vert U_{jj}/ U_{jk} \vert$. From~\cite{lemeire1975bounds}, we
know that $\min_j \vert U_{jj} \vert / \max_{i,j} \vert U_{ij} \vert \ge
1/\kappa(U)$, where $\kappa(U)$ denotes the condition number of $U$. Since
$\kappa(U)=\sqrt{\sigma_1/\sigma_d}$, we have then that $\min_{k>j} \vert
U_{jj}/ U_{jk} \vert \ge \sqrt{\sigma_d/\sigma_1}$. Therefore, the cost of
the adaptive Gauss-Legendre quadrature to evaluate $I_j$ will grow in
proportion to $\log(\sigma_1/\sigma_d)$.

When $\tilde z_j$ is close to $\R$ and $\Re{\tilde z_j} \in [a_j,b_j]$, the
integral is nearly singular and the cost of adaptive Gauss-Legendre
quadrature is prohibitive. Instead, we proceed by first dividing the
interval $[a_j,b_j]$ into the three subintervals $[a_j,\Re{\tilde
z_j}-\delta_j]$, $[\Re{\tilde z_j}-\delta_j,\Re{\tilde z_j}+\delta_j]$, and
$[\Re{\tilde z_j}-\delta_j,b_j]$, where $\delta_j > 0$. The parameter
$\delta_j$ is chosen
so that the function 
\begin{align*}
\psi_j(\xi_j) = e^{(\alpha_j\xi_j+\beta_j)^2-w_j^2}f(\xi_j - \tilde z_j)(\xi_j -
\tilde z_j)
\tilde\phi^{(j)}(\xi_{k+1},\xi_{k+2},\ldots,\xi_{j-1},\xi_j,\tilde
z_{j+1},\ldots,\tilde z_d),
\end{align*}
which is analytic in a neighborhood of $\Re{\tilde z_j}$, can be
approximated to precision $\varepsilon$ on the interval 
$[\Re{\tilde z_j}-\delta_j,\Re{\tilde z_j}+\delta_j]$
by a $N$th-order Chebyshev interpolant, with, for example, $N=20$. 
Since the function $\psi_j$ has features on the scale of 
$\min_{k>j} \vert U_{jj}/ U_{jk} \vert$, $\delta_j$ must be chosen in proportion
to these features, as, for example, $\delta_j = (1/6)\cdot (b_j-a_j)\min_{k>j}
\vert U_{jj}/ U_{jk} \vert$.
As discussed in~\cite{shen2022polynomial}, the fact that $\psi_j$ is
accurately approximated by its $N$th order Chebyshev interpolant 
means that, in practice, $\psi_j$ can be approximated uniformly to precision
$\varepsilon$ by a monomial expansion, where the coefficients of the monomial
expansion are computed by solving a Vandermonde system using a backward
stable solver. Once the approximation
\begin{align*}
\sum_{n=0}^N a_n (\xi_j - \Re{\tilde z_j})^n \approx \psi_j(\xi_j),
\end{align*}
has been computed, the integral
\begin{align*}
\int_{\Re{\tilde z_j}-\delta_j}^{\Re{\tilde z_j}+\delta_j}
\frac{\psi_j(\xi_j)}{\xi_j - \tilde z_j - i0^\pm} \, d\xi_j
\end{align*}
is evaluated from the moments
\begin{align*}
\int_{\Re{\tilde z_j}-\delta_j}^{\Re{\tilde z_j}+\delta_j}
\frac{(\xi_j - \Re{\tilde z_j})^n}{\xi_j - \tilde z_j - i0^\pm} \, d\xi_j,
\end{align*}
which are computed analytically using well-known recurrence relations (see,
for example~\cite{helsing2008evaluation}). When $\Re{\tilde z_j}=0$,
the correct branch corresponding to either $i0^+$ or $i0^-$ is selected 
by rotating the branch cut analytically. The integrals over
the intervals $[a_j,\Re{\tilde z_j}-\delta_j]$ and $[\Re{\tilde
z_j}-\delta_j,b_j]$, when these intervals are not empty, are evaluated
quickly and accurately by adaptive Gauss-Legendre quadrature, since away
from $\Re{\tilde z_j}$ the integrand is smooth.
Finally, we observe that, when $\min_{k>j} \vert U_{jj}/ U_{jk} \vert \le
1$, no fine-scale features are present, and adaptive Gauss-Legendre
quadrature can be replaced with quadratures of fixed orders for greater
efficiency.

% The following technique can be used to accelerate the evaluation of $I_d$.
We observe that the integrand appearing
in the formula for $I_d$ involves only the function $f$, and so has no
possibility of fine-scale features.  This fact, combined with the
availability of highly optimized numerical codes for evaluating the
classical Faddeeva $w$ function, can be leveraged to dramatically accelerate
the evaluation of $I_d$. Recalling that $f\colon \C\setminus
(\Z/t) \to \C$ is meromorphic with residue $(2\pi i)^{-1}$ at the elements of
$\Z/t$, we define the function
\begin{align*}
\Phi_d(\xi_d) = f(\xi_d - \tilde z_d) - \frac{1}{2\pi i}\sum_{j=-n}^n
\frac{1}{\xi_d - \tilde z_d - j/t},
\end{align*}
which is analytic on the set $\C \setminus (\{ j\in \Z : \vert j \vert > n
\}/t)$. Next, we approximate $\Phi_d$ on the interval $[a_d,b_d]$ by using
Cauchy's theorem combined with the trapezoidal rule. Letting $r =
b_d-a_d$, we set $\zeta_j = (a_d+b_d)/2 + re^{(2\pi j/m + \delta)i}$ and
$\delta\zeta_j =(2\pi i/m)r e^{(2\pi j/m + \delta)i}$, for $j=1,2,\ldots,m$,
where $\delta \in \R$ is chosen so that none of the points $\zeta_j$
coincide with $\tilde z_j$. We then have that
\begin{align*}
\Phi_d(\xi_d) \approx \sum_{j=1}^m
\frac{\Phi_d(\zeta_j)}{\xi_d-\zeta_j}\delta\zeta_j.
\end{align*}
Combining these two formulas,
\begin{align*}
\begin{split}
I_d 
&\approx \sum_{j=1}^m  \Phi_d(\zeta_j) \delta\zeta_j \int_\R
\frac{e^{-\eta_d(\bm{\tilde z}^{(d)})^2 - \eta_d(\bm{\bar z})^2}}{\xi_d
- \zeta_j} \, d\xi_d \\
&\hspace*{6em} + \frac{1}{2\pi i}\sum_{j=-n}^n \int_\R
  \frac{e^{-\eta_d(\bm{\tilde z}^{(d)})^2 -
\eta_d(\bm{\bar z})^2}}{\xi_d - \tilde z_d - j/t} \, d\xi_d.
\end{split}
\end{align*}
The integrals appearing in this formula can be expressed in terms of the
classical Faddeeva~$w$ function, which can be evaluated rapidly using, for
example, the highly optimized implementation provided
in~\cite{faddeevapackage}. In practice, we find $n=6$ and $m=18$ sufficient
to obtain full machine precision accuracy for the function
$f(z)=t(1-e^{-2\pi itz})^{-1}$.

\subsection{Evaluation of the singular Hadamard finite-part integral}
\label{sec:finitepart}

The final step in evaluating the set zeta function for the corner is
the evaluation of the Hadamard finite part integral
\begin{align*}
\ddashint_0^{1/\lambda} t^{\nu-d-1} \psi_{L-\bm x}(t,\bm y)\, \dd t.
\end{align*}
We proceed by splitting this integral into two parts,
\begin{align*}
\ddashint_0^c t^{\nu-d-1} \psi_{L-\bm x}(t,\bm y)\, \dd t  +
\int_c^{1/\lambda} t^{\nu-d-1} \psi_{L-\bm x}(t,\bm y)\, \dd t.
\end{align*}
In order to evaluate this first integral, we form the power series approximation
\begin{align*}
\psi_{L-\bm x}(t,\bm y) \approx \sum_{j=0}^m a_j (t/c)^j,
\end{align*}
so that
\begin{align*}
\ddashint_0^c t^{\nu-d-1} \psi_{L-\bm x}(t,\bm y)\, \dd t 
\approx c^{\nu-d} \sum_{j=0}^m \frac{a_j}{\nu-d+j},
\end{align*}
where we compute the coefficients $a_j$, $0\le j\le m$, by the formula
\begin{align*}
a_j = \frac{1}{2\pi} \int_0^{2\pi} \psi_{L-\bm x}(ce^{i\theta},\bm y)
e^{-ij\theta}\, \dd \theta
\approx
\frac{1}{m} \sum_{k=0}^{m-1} \psi_{L-\bm x}(ce^{2\pi i k/m},\bm y) 
e^{-2\pi ijk/m}.
\end{align*}
One problem with this formulation, however, is that $\psi_{L-\bm x}(t,\bm
y)$ can become extremely large for complex values of $t$, resulting in a
large cancellation error. Recall that
\begin{align*}
\psi_{L-\bm x}(t,\bm y) = 
e^{2\pi i \bm x \cdot \bm y}  e^{-\pi t^2 {\bm x}^2}
\frac{1}{V_\Lambda}
w(\bm Y(t); \otimes_{j=1}^d t(1-e^{-2\pi it(\cdot-i0^+)})^{-1}, A^T A/\pi)
\end{align*}
and
\begin{align*}
\bm Y(t)= -A^T (\bm y/t+i t \bm x).
\end{align*}
In the previous section, we defined
\begin{align*}
\phi(\bm Y(t)) = w(\bm Y(t); \otimes_{j=1}^d t(1-e^{-2\pi
it(\cdot-i0^+)})^{-1}, A^T A/\pi),
\end{align*}
as well as a scaled version $\tilde \phi(\bm Y(t)) = e^{-\beta} \phi(\bm
Y(t))$, so that $\tilde \phi(\bm Y(t)) \lesssim 1$. When $t>0$ is real, it
is generally the case that $e^{-\pi t^2 x^2 + \beta} \lesssim 1$, since $\pi
t^2 x^2 = \pi (A^{-T}\Im{\bm Y(t)})^2$, however, for complex $t$, this
relationship no longer holds. In order to evaluate the coefficients
$a_j$, $0\le j\le m$, stably, we must ensure that the integrand
remains small. 

To proceed, we will first need to modify the definitions of the number
$\bm{\bar z}$ and the functions $\tilde \phi^{(k)}\colon \C^d\to \C$, $0\le
k\le d-1$, introduced in the previous section.
Suppose that $\bm \gamma \in \{0,1\}^d$, and let $\bm{\bar z}_{\bm \gamma} =
\bm{\bar z}\odot \bm\gamma$,  where $\odot$ represents componentwise
multiplication, with $\bm{\bar z}$ defined as before. Let $\beta_{\bm\gamma}
= \sum_{j=1}^n \eta_j(\bm{\bar z}_{\bm\gamma})$. Next, let the functions
$\tilde \phi_{\bm\gamma}^{(k)} \colon \C^d \to \C$, $0\le k\le d-1$, be
defined by the formula
\begin{align*}
\begin{split}
&{\tilde\phi}_{\bm\gamma}^{(k)}(\bm{\tilde z}^{(k)}) = \int_{a_{k+1}}^{b_{k+1}} 
e^{-\eta_{k+1}(\bm{ \tilde z}^{(k+1)})^2-\eta_{k+1}(\bm{\bar z}_{\bm\gamma})^2}
\\
&\hspace*{9.5em}
f(\xi_{k+1} - \tilde z_{k+1} - i0^+) {\tilde\phi}^{(k+1)}(\bm{\tilde
z}^{(k+1)}) \, \dd \xi_{k+1},
\end{split}
\end{align*}
when either $\Im{\tilde z_{k+1}} \ge 0$ or $\gamma_{k+1} = 0$, and let
\begin{align*}
\begin{split}
&{\tilde\phi}_{\bm\gamma}^{(k)}(\bm{\tilde z}^{(k)}) = 
e^{-\eta_{k+1}(\bm{\tilde z}^{(k)})^2
-\eta_{k+1}(\bm{\bar z}_{\bm\gamma})^2}
{\tilde\phi}^{(k+1)}(\bm{\tilde z}^{(k)} + \Delta_{k+1}(\bm{\tilde z}^{(k)})) \\ 
&\hspace*{4.5em} +
\int_{a_{k+1}}^{b_{k+1}}
e^{-\eta_{k+1}(\bm{\tilde z}^{(k+1)})^2-\eta_{k+1}(\bm{\bar z}_{\bm\gamma})^2} \\
&\hspace*{9.5em} 
f(\xi_{k+1} - \tilde z_{k+1} -i0^-) {\tilde\phi}^{(k+1)}(\bm{\tilde z}^{(k+1)}) \,
\dd \xi_{k+1},
\end{split}
\end{align*}
when $\Im{z_{k+1}} < 0$ and $\gamma_{k+1} = 1$, where $\bm{\tilde z}^{(k)} =
(\xi_1,\xi_2,\ldots,\xi_k,\tilde z_{k+1},\tilde z_{k+2},\ldots,\tilde z_d)$,
$\bm z\in \C^d$, and $\tilde\phi_{\bm\gamma}^{(d)}(\bm \xi) = 1$. 
Suppose that, for each $j$ for which $\gamma_j = 0$, we have that
$\abs{\Im{z_j}} \le H_1$, where the set on which $e^{\bm\xi^T G^{-1}\bm \xi}
\le 1$ is contained in $[-H_1,H_1]^d$, and that $\abs{\Re{z_j}} \ge
\bar H_\varepsilon$, where the support of $e^{-\bm\xi^T G^{-1}\bm \xi}$ to precision
$\varepsilon$ is contained in $[-\bar H_\varepsilon,\bar H_\varepsilon]^d$.
Then, it is easy to see that $\tilde \phi^{(0)}_{\bm\gamma}(\bm{\tilde z})
\approx \tilde \phi^{(0)}(\bm{\tilde z})$, to precision $\varepsilon$.

Consider now $\bm{\tilde z} = \bm Y(t)$, where $A^T \bm y \in [-1/2,1/2)^d$.
Let $\bm\gamma\in \{0,1\}^d$ and suppose that $\gamma_j=1$ whenever $(A^T\bm
y)_j = 0$ and $\gamma_j=0$ otherwise, for $1\le j\le d$.  Then, for
sufficiently small $t > 0$, we have that $\tilde
\phi^{(0)}_\gamma(\bm{Y}(t)) \approx \tilde \phi^{(0)}(\bm{Y}(t))$, to
precision $\varepsilon$.
Notice that $\phi^{(0)}(\bm{Y}(ce^{i\theta}))$
can become very large for $0\le \theta < 2\pi$, even for small values of $c>0$, 
while $\phi^{(0)}_{\bm\gamma}(\bm{Y}(ce^{i\theta})) =
e^{\beta_{\bm\gamma}}\tilde \phi^{(0)}_{\bm\gamma}(\bm{Y}(ce^{i\theta}))
\lesssim 1$ for all $0\le \theta < 2\pi$, provided that $c>0$ is
sufficiently small.

The parameter $c>0$ can be chosen as follows. First,
letting $\sigma_1\ge \sigma_2 \ge \cdots \ge \sigma_d > 0$ denote the
eigenvalues of $A^TA$, it is easy to see that the support of $e^{-\pi
(A^{-T} \bm\xi)^2}$ is, to precision $\varepsilon$, contained in
$[-\bar H_\varepsilon,\bar H_\varepsilon]^d$ where $\bar
H_\varepsilon=\sqrt{\sigma_1 \log(1/\varepsilon)/\pi}$. Likewise, we see
that $e^{\pi (A^{-T} \bm\xi)^2} \le 1$ on the set $[-H_1,H_1]^d$
where $H_1=\sqrt{\sigma_d/\pi}$.  In order to use $\tilde
\phi^{(0)}_{\bm\gamma}(\bm{Y}(ce^{i\theta}))$ to compute the coefficients of our
power series approximation, this function must be periodic
and smooth in $\theta$, which will be the case only if
$(\bm{Y}(ce^{i\theta}))_j$ avoids the branch cut $[a_j,b_j]$, for all $1\le
j\le d$ for which $\gamma_j=0$. 
We now observe that $\abs{(\bm{Y}(ce^{i\theta}))_j}
\ge \abs{(A^T\bm y)_j}/c - c \abs{(A^T\bm x)_j}$, for all $1\le j\le d$,
provided that $c$ is sufficiently small.
Therefore, the branch cuts $[a_j,b_j]$ are avoided by
$(\bm{Y}(ce^{i\theta}))_j$, for all $1\le j\le d$ for which $\gamma_j=0$,
over all $0\le \theta < 2\pi$, provided that $\abs{(A^T\bm y)_j}/c - c
\abs{(A^T\bm x)_j} \ge \bar H_\varepsilon$, for all $1\le j\le d$ for which
$\gamma_j=0$.  
In order for $\phi^{(0)}_{\bm\gamma}(\bm{Y}(ce^{i\theta}))$ to remain
small for $0\le \theta <2\pi$, we also require that $c\abs{(A^T\bm x)_j} \le
H_1$, for all $1\le j\le d$ for which $\gamma_j=1$.
Finally, to ensure both that $c\abs{(A^T\bm x)_j} \le H_1$ for all
$1\le j\le d$ and that $e^{-\pi t^2\bm x^2}$ is small, we require that $c\le
1/(\abs{\bm x}\sqrt{\pi})$. The largest value of $c \le 1/\lambda$ which
satisfies these conditions is then selected. 

Note that, when $t$ is complex, the special scheme described in the previous
section for accelerating the evaluation of the inner-most integral using
the classical Faddeeva function, cannot be used.
Once the Hadamard finite part integral over $[0,c]$ is evaluated, the
remaining integral over the interval $[c,1/\lambda]$ is then evaluated using
a standard adaptive integration scheme.

\section{Numerical experiments}\label{sec:experiments}

In this section, we present several numerical experiments demonstrating the
accuracy and efficiency of our algorithm.  We check the accuracy of the
algorithm using known formulas for specific infinite discrete sets in
\cref{sec:knownform} and test the algorithm on subsets
of general lattices in \cref{sec:genlat}.

\subsection{Comparison to known formulas}
  \label{sec:knownform}

There exist known analytical formulas for zeta functions over certain
infinite discrete sets. In this section, we evaluate the accuracy of our
algorithm by evaluating these zeta functions, and comparing our results with
these formulas. We consider the sums
\begin{alignat*}{2}
S_1 &= \sum_{m,n\in \N} (m^2+n^2)^{-\nu/2}, \\
S_2 &= \sum_{m,n\in \N} (m^2+n^2)^{-\nu/2} (-1)^{m+n}, \\
S_3 &= \sum_{m,n\in \N} (m^2+n^2)^{-\nu/2} (-1)^{n-1}, \\
S_4 &= \hspace*{-0.5em}\sum_{k,\ell\in 2\N-1} (k^2+\ell^2)^{-\nu/2}, \\
S_5 &= \sum_{\substack{m\in \N\\ p \in 2\N}} (m^2+p^2)^{-\nu/2}, \\
S_6 &= \hspace*{-1.5em}\sum_{k_1,k_2,k_3,k_4 \in 2\N-1}\hspace*{-1.5em}
  (k_1^2+k_2^2+k_3^2+k_4^2)^{-\nu/2}, \\
S_7 &= 4\hspace*{-0.5em}\sum_{\ell,m,n \in \N}
  \big((\ell-\tfrac{1}{2})^2+m^2+n^2\big)^{-\nu/2} (-1)^{m}.
\end{alignat*}
Analytical formulas expressing these sums in terms of the Riemann zeta
function $\zeta(s)$ and the functions
\begin{align*}
\beta(s) = \sum_{n=0}^\infty (-1)^n (2n+1)^{-s},
\end{align*}
\begin{align*}
A(s) = \sum_{n=0}^\infty (-1)^n (4n+1)^{-s},
\end{align*}
and
\begin{align*}
B(s) = \sum_{n=1}^\infty (-1)^{n+1} (4n-1)^{-s},
\end{align*}
are presented by Glasser
in~\cite{glasser1973evaluation,glasser1973evaluation2} with corrections by
Zucker in \cite{zucker2017exact}.  We summarize these results in
\cref{tab:glassersums}.

\begin{table}[h]
  % \centering\small
  % \setlength{\tabcolsep}{5pt}
  \centering
\begin{tabular}{cl}
Sum & Value \\
\midrule
$S_1$ & $\zeta(\frac{\nu}{2})\beta(\frac{\nu}{2}) - \zeta(\nu)$ \\
\addlinespace[.20em]
$S_2$ & $(2^{1-\nu/2} - 1)\zeta(\frac{\nu}{2})\beta(\frac{\nu}{2}) 
+ (1-2^{1-\nu})\zeta(\nu)$ \\
\addlinespace[.20em]
$S_3$ & $(2^{-\nu/2}-2^{1-\nu})\zeta(\frac{\nu}{2})\beta(\frac{\nu}{2}) +
2^{-\nu}\zeta(\nu)$ \\
\addlinespace[.20em]
$S_4$ & $2^{-\nu/2}(1-2^{-\nu/2})\beta(\frac{\nu}{2})\zeta(\frac{\nu}{2})$ \\
\addlinespace[.20em]
$S_5$ & $\tfrac{1}{2}(1-2^{-\nu/2}+2^{1-\nu})\zeta(\frac{\nu}{2})
  \beta(\frac{\nu}{2}) - \tfrac{1}{2}(1+2^{-\nu})\zeta(\nu)$ \\
\addlinespace[.20em]
$S_6$ & $2^{-\nu}(1-2^{-\nu/2})(1-2^{1-\nu/2})\zeta(\frac{\nu}{2})
  \zeta(\frac{\nu}{2}-1)$ \\
\addlinespace[.20em]
$S_7$ & $2^{\nu}\bigl(\beta(\nu-1) - A(\frac{\nu}{2})^2 + B(\frac{\nu}{2})^2
- (1-2^{-\nu/2})\zeta(\frac{\nu}{2})\beta(\frac{\nu}{2})$ \\
& $\hspace*{4em} + (1-2^{-\nu})\zeta(\nu)\bigr)$ \\
\addlinespace[.55em]
\end{tabular}

\caption{
  \label{tab:glassersums} Analytical formulas provided
  in~\cite{glasser1973evaluation,glasser1973evaluation2,zucker2017exact} expressing certain
  lattice sums in terms of the functions $\zeta(s)$, $\beta(s)$,
  $A(s)$, and $B(s)$.
}
\end{table}

The sums $S_1,S_2,\ldots,S_7$ can be easily expressed as evaluations of
$Z_{A\Nz,\nu}\babs{{\!\bm x\atop \!\bm y}\!}$, for the values of $A$, $\bm x$, and
$\bm y$ provided in \cref{tab:glasserzeta}.

\begin{table}[h]
  % \centering\small
  % \setlength{\tabcolsep}{5pt}
  \centering
\begin{tabular}{cccc}
Sum & $\bm x^T$ & $\bm y^T$ & $A$ \\
\midrule
$S_1$ & $(-1,-1)$ & $(0,0)$ & $\left(\!\!\begin{array}{cc}
1 & 0 \\
0 & 1
\end{array}\!\!\right)$ \\
\addlinespace[.40em]
$S_2$ & $(-1,-1)$ & $(\tfrac{1}{2},\tfrac{1}{2})$ &
$\left(\!\!\begin{array}{cc}
1 & 0 \\
0 & 1
\end{array}\!\!\right)$ \\
\addlinespace[.40em]
$S_3$ & $(-1,-1)$ & $(\tfrac{1}{2},0)$ &
$\left(\!\!\begin{array}{cc}
1 & 0 \\
0 & 1
\end{array}\!\!\right)$ \\
\addlinespace[.40em]
$S_4$ & $(-1,-1)$ & $(0,0)$ &
$\left(\!\!\begin{array}{cc}
2 & 0 \\
0 & 2
\end{array}\!\!\right)$ \\
\addlinespace[.40em]
$S_5$ & $(-1,-2)$ & $(0,0)$ &
$\left(\!\!\begin{array}{cc}
1 & 0 \\
0 & 2
\end{array}\!\!\right)$ \\
\addlinespace[.40em]
$S_6$ & $(-2,-2,-2,-2)$ & $(0,0)$ &
$\left(\!\!\begin{array}{cccc}
1 & 0 & 0 & 0 \\
0 & 1 & 0 & 0 \\
0 & 0 & 1 & 0 \\
0 & 0 & 0 & 1
\end{array}\!\!\right)$ \\
\addlinespace[.40em]
$S_7$ & $(-\tfrac{1}{2},-1,-1)$ & $(0,\tfrac{1}{2},0)$ &
$\left(\!\!\begin{array}{ccc}
1 & 0 & 0 \\
0 & 1 & 0 \\
0 & 0 & 1
\end{array}\!\!\right)$ \\
\addlinespace[.55em]
\end{tabular}

\caption{
  \label{tab:glasserzeta} The values of $\bm x$, $\bm y$, and $A$ for
  expressing certain lattice sums in terms of the function
  $Z_{A\Nz,\nu}\babs{{\!\bm x\atop \!\bm y}\!}$.
}
\end{table}

We evaluate our algorithm by comparing our numerically computed results with
the formulas provided in \cref{tab:glassersums}, where the formulas are
evaluated using Mathematica to 200~digits of precision. We report
$E=\min(E_\text{abs},E_\text{rel})$, where $E_\text{abs}$ is the absolute
error and $E_\text{rel}$ is the relative error, over a range of~$\nu$
from $-2+10^{-4}$ to $10+10^{-4}$, with
values sampled in increments of $1/80$. The results are presented in
\cref{fig:glasser}. We do not report the timings, however, the entire experiment
was carried out in less than 10 minutes on a laptop with 16~GB of RAM and an
Intel 8th~Gen.\,Core i7-8550U~CPU.

\begin{figure}
    \centering
    \begin{subfigure}[t]{.49\textwidth}
      \caption{$S_1$}
      \includegraphics[width=\textwidth]{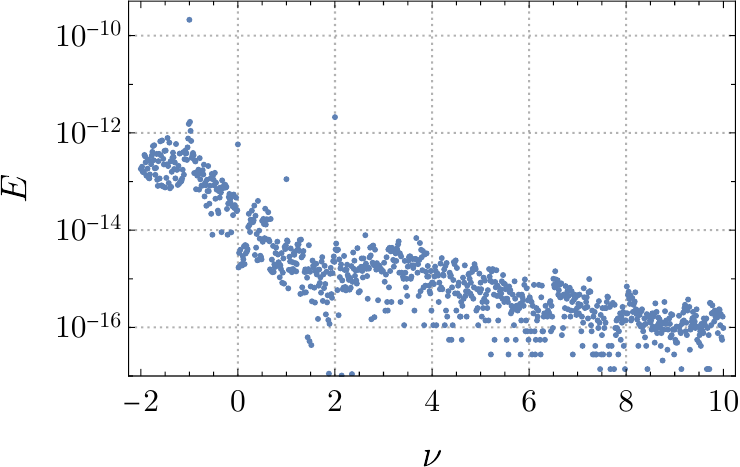}
    \end{subfigure}
    \begin{subfigure}[t]{.49\textwidth}
      \caption{$S_2$}
      \includegraphics[width=\textwidth]{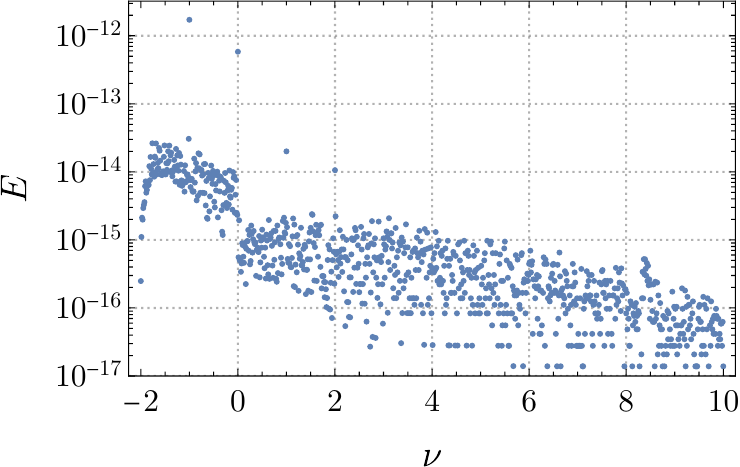}
    \end{subfigure}
    \begin{subfigure}[t]{.49\textwidth}
      \caption{$S_3$}
      \includegraphics[width=\textwidth]{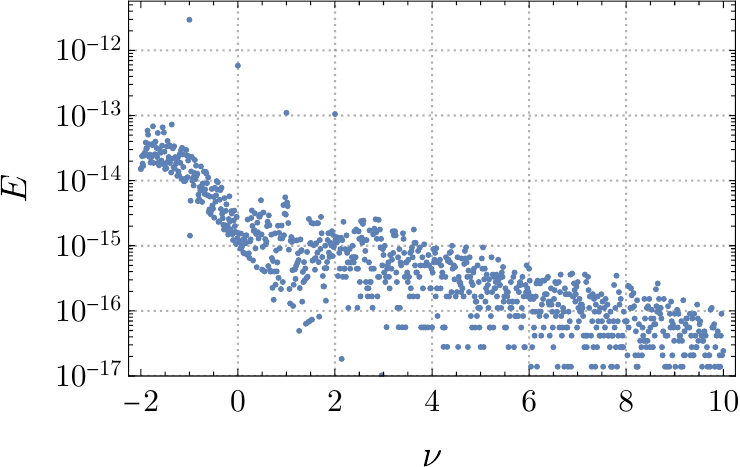}
    \end{subfigure}
    \begin{subfigure}[t]{.49\textwidth}
      \caption{$S_4$}
      \includegraphics[width=\textwidth]{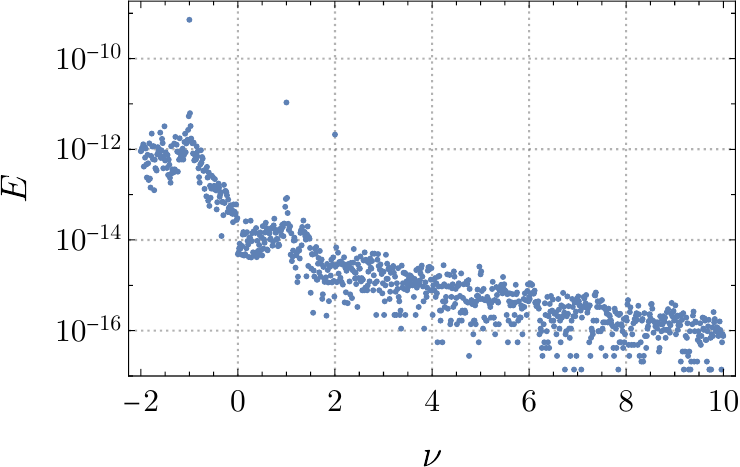}
    \end{subfigure}
    \begin{subfigure}[t]{.49\textwidth}
      \caption{$S_5$}
      \includegraphics[width=\textwidth]{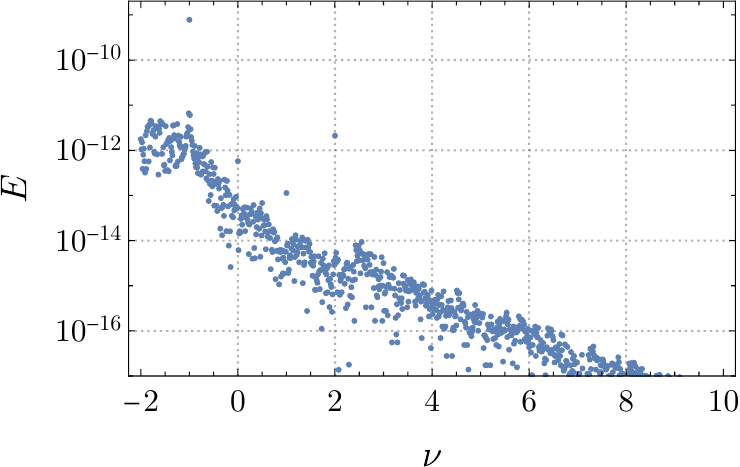}
    \end{subfigure}
    \begin{subfigure}[t]{.49\textwidth}
      \caption{$S_6$}
      \includegraphics[width=\textwidth]{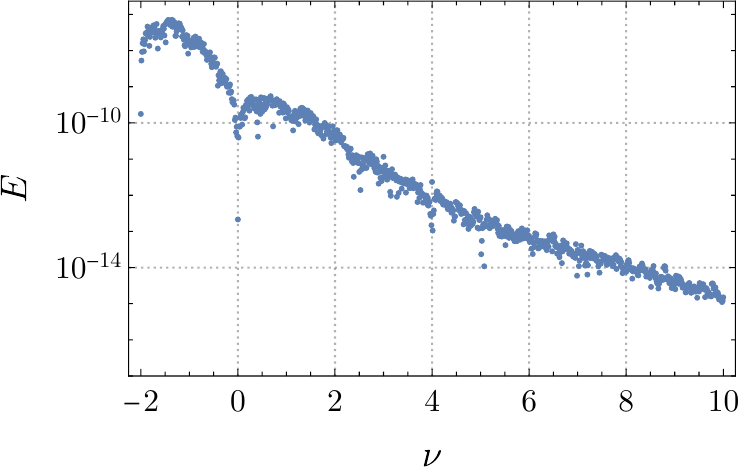}
    \end{subfigure}
    \begin{subfigure}[t]{.49\textwidth}
      \caption{$S_7$}
      \includegraphics[width=\textwidth]{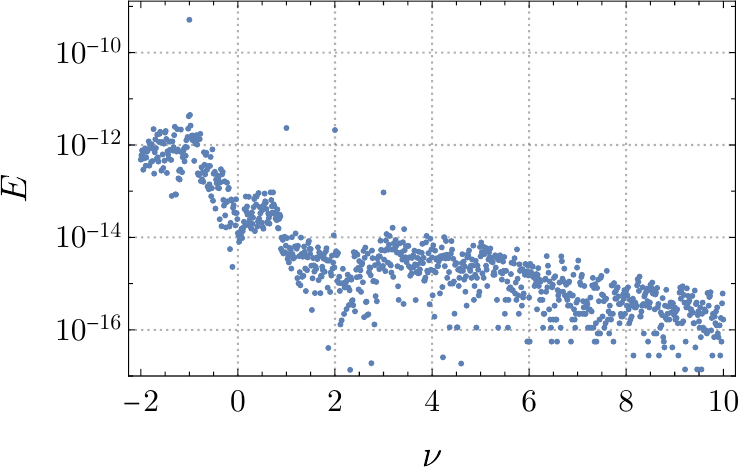}
    \end{subfigure}
    \caption{The error $E=\min(E_\text{abs},E_\text{rel})$ of our algorithm
    for certain lattice sums.}
    \label{fig:glasser}
\end{figure}

\subsection{Subsets of general lattices}
  \label{sec:genlat}

In this section, we apply our algorithm to subsets of general lattices of
the form $\Lambda = A \Z^d$.  In order to provide a concise description of
the lattice matrix $A$, we use the following crystallographic conventions.
Suppose that $d=3$, and let $\bm a$, $\bm b$, and $\bm c$ denote the columns
of $A$. Let $a=\vert \bm a \vert$, $b=\vert \bm b \vert$, and $c=\vert \bm
c\vert$, and let $\alpha$ denote the angle between $\bm b$ and $\bm c$,
$\beta$ denote the angle between $\bm a$ and $\bm c$, and $\gamma$ denote
the angle between $\bm a$ and $\bm b$ (see \cref{fig:lattice}).  The
parameters $a$, $b$, $c$, $\alpha$, $\beta$, $\gamma$ are known as the
lattice parameters.  For $d=2$, the lattice is summarized by lattice
parameters $a$, $b$, and $\gamma$, while for $d=1$, the lattice is
summarized by just $a$. In this section, we consider the zeta
function over the finite sublattice $L=\Lambda \cap \bigl(A \prod_{j=1}^d
[0,m]\bigr)$, where $m$ is a nonnegative even integer. The lattice matrices
considered in this section are shown in
\cref{tab:lattices}.

\begin{figure}
    \centering
    \includegraphics[width=.32\textwidth]{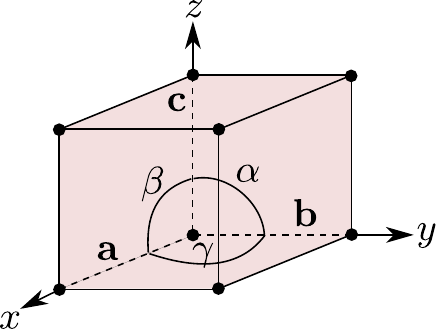}
    \caption{The lattice parameters indicated on the unit cell $E_\Lambda$.}
    \label{fig:lattice}
\end{figure}

The experiments were performed on Erlangen National High Performance
Computing Center (NHR@FAU), on the Fritz parallel CPU cluster.  Each
individual zeta function evaluation was performed using a single
computational thread.  We implemented our algorithm using a combination of
Fortran~77 and Fortran~90, and compiled it using the \texttt{gfortran}
compiler, version~8.5.0, with the \texttt{-03} and \texttt{-march=native}
flags.

We first tested our algorithm over a wide range of parameters $\bm x$ and
$\bm y$.  Letting $S=\prod_{j=1}^d\{ -m/2, 0, m/2\} + \prod_{j=1}^d \{
-\frac{1}{2},0,\frac{1}{2} \}$, we evaluated the zeta function at the points
$\bm x \in A S$.  Similarly, we chose $\bm y \in A^{-T} \prod_{j=1}^d
\{0,1/3,2/3 \}$.  The results of our algorithm are shown in
\cref{tab:results} for $\nu=d+0.1$, in which we use the following notation.
\begin{itemize}

\item $N$: The total numbers of points in the lattice: $(m+1)^d$.

\item $E_\text{max}$: The maximum, over all points $\bm x$ and $\bm y$, of
the minimum of the absolute error~$E_\text{abs}$ and relative
error~$E_\text{rel}$, $E_\text{max} = \max_{\bm x,\bm y}
\min(E_\text{abs},E_\text{rel})$.

\item $t_\text{min}$, $t_\text{max}$: The minimum and maximum CPU time
required for the evaluation of the set zeta function using our
algorithm, over all points~$\bm x$ and~$\bm y$.

\item $t_\text{avg}$: The average CPU time required for the evaluation of
the set zeta function using our algorithm, over all points
$\bm x$ and $\bm y$.

\item $t_\text{naive}$: The maximum CPU time required for the evaluation of
the set zeta function using our naive summation, over all points $\bm x$ and
$\bm y$.

\end{itemize}

\begin{table}[h]
  % \centering\small
  % \setlength{\tabcolsep}{5pt}
  \centering
\begin{tabular}{cccccccc}
Matrix & $d$ & $a$ & $b$ & $c$& $\alpha$ & $\beta$ & $\gamma$ \\
\midrule
$A_1$ & $d=1$ & $1.1$ & -- & -- & -- & -- & -- \\
$A_2$ & $d=2$ & $1.1$ & $1.2$ & -- & -- & -- & $2\pi/3$ \\
$A_3$ & $d=2$ & $1.1$ & $1.2$ & -- & -- & -- & $\pi/2$ \\
$A_4$ & $d=3$ & $1.1$ & $1.2$ & $1.3$ & $\pi/2$ & $\pi/2$ & $2\pi/3$ \\
$A_5$ & $d=3$ & $1.1$ & $1.2$ & $1.3$ & $\pi/2$ & $\pi/2$ & $\pi/2$ \\
\addlinespace[.55em]
\end{tabular}

\caption{
  \label{tab:lattices} The lattice matrices considered in our experiments.
}
\end{table}

\begin{table}[h]
\small
  % \centering\small
  % \setlength{\tabcolsep}{5pt}
  \centering
\begin{tabular}{cllllll}
Matrix & $m$ & $N$ & $E_\text{max}$ & $t_\text{min}$--$t_\text{max}$ &
$t_\text{avg}$ & $t_\text{naive}$ \\
\midrule
\multirow{6}*{$A_1$} 
& 1.00\e{2} &  1.00\e{2}   & 5.79\e{-15} & 1.49\e{-2}--6.04\e{-2} & 2.96\e{-2} & 1.02\e{-3} \\
& 1.00\e{4} &  1.00\e{4}   & 2.88\e{-15} & 1.48\e{-2}--5.99\e{-2} & 2.89\e{-2} & 3.20\e{-2} \\
& 1.00\e{6} &  1.00\e{6}   & 2.52\e{-15} & 1.01\e{-2}--4.11\e{-2} & 1.98\e{-2} & 2.15\e{0} \\
& 1.00\e{8} &  1.00\e{8}   & 2.49\e{-15} & 1.01\e{-2}--4.10\e{-2} & 1.98\e{-2} & 2.15\e{2} \\
& 1.00\e{10} & 1.00\e{10}  & --          & 9.83\e{-3}--4.25\e{-2} & 2.03\e{-2} & -- \\
& 1.00\e{14} & 1.00\e{14}  & --          & 9.83\e{-3}--4.06\e{-2} & 1.95\e{-2} & -- \\
\midrule
\multirow{5}*{$A_2$} 
& 1.00\e{1} & 1.21\e{2}  & 4.51\e{-14} & 1.01\e{0}-- 4.04\e{1} & 6.34\e{0} & 3.94\e{-3} \\
& 1.00\e{2} & 1.02\e{4}  & 5.33\e{-14} & 9.75\e{-1}--3.19\e{1} & 5.37\e{0} & 5.57\e{-2} \\
& 1.00\e{3} & 1.00\e{6}  & 3.95\e{-14} & 9.65\e{-1}--2.73\e{1} & 4.88\e{0} & 5.61\e{0} \\
& 1.00\e{4} & 1.00\e{8}  & 3.79\e{-14} & 9.49\e{-1}--2.85\e{1} & 4.96\e{0} & 3.35\e{2}\\
& 1.00\e{5} & 1.00\e{10} & --          & 9.61\e{-1}--3.65\e{1} & 5.48\e{0} & -- \\
& 1.00\e{7} & 1.00\e{14} & --          & 9.62\e{-1}--3.73\e{1} & 5.55\e{0} & -- \\
\midrule
\multirow{5}*{$A_3$} 
& 1.00\e{1} & 1.21\e{2}  & 2.86\e{-14} & 4.27\e{-2}--1.67\e{0}  & 2.46\e{-1} & 3.32\e{-3} \\
& 1.00\e{2} & 1.02\e{4}  & 5.83\e{-14} & 3.69\e{-2}--4.11\e{-1} & 1.26\e{-1} & 5.59\e{-2} \\
& 1.00\e{3} & 1.00\e{6}  & 5.92\e{-14} & 3.65\e{-2}--3.89\e{-1} & 1.26\e{-1} & 3.76\e{0} \\
& 1.00\e{4} & 1.00\e{8}  & 4.29\e{-14} & 3.60\e{-2}--5.46\e{-1} & 1.27\e{-1} & 3.53\e{2} \\
& 1.00\e{5} & 1.00\e{10} & --          & 3.64\e{-2}--7.33\e{-1} & 1.63\e{-1} & --         \\
& 1.00\e{7} & 1.00\e{14} & --          & 3.64\e{-2}--5.74\e{-1} & 1.41\e{-1} & --         \\
\midrule
\multirow{5}*{$A_4$} 
& 4.00\e{0} & 1.25\e{2}  & 6.41\e{-13} & 5.25\e{0}--1.49\e{2} & 2.64\e{1} & 1.34\e{-2} \\
& 2.20\e{1} & 1.22\e{4}  & 3.24\e{-13} & 4.84\e{0}--1.54\e{2} & 2.60\e{1} & 1.54\e{-1} \\
& 1.00\e{2} & 1.03\e{6}  & 2.83\e{-13} & 2.02\e{0}--1.34\e{2} & 1.87\e{1} & 7.10\e{0} \\
& 4.64\e{2} & 1.01\e{8}  & --          & 2.02\e{0}--1.43\e{2} & 1.69\e{1} & -- \\
& 2.20\e{3} & 1.07\e{10} & --          & 2.02\e{0}--1.81\e{2} & 1.63\e{1} & -- \\
& 4.64\e{4} & 9.99\e{13} & --          & 2.01\e{0}--1.94\e{2} & 1.73\e{1} & -- \\
\midrule
\multirow{5}*{$A_5$} 
& 4.00\e{0} & 1.25\e{2}  & 2.90\e{-13} & 1.25\e{0}-- 2.67\e{1} & 4.82\e{0}  & 1.85\e{-2} \\
& 2.20\e{1} & 1.22\e{4}  & 1.87\e{-13} & 9.38\e{-1}--2.07\e{1} & 4.22\e{0}  & 1.78\e{-1} \\
& 1.00\e{2} & 1.03\e{6}  & 1.73\e{-13} & 1.05\e{-1}--1.42\e{1} & 1.46\e{0}  & 8.19\e{0} \\
& 4.64\e{2} & 1.01\e{8}  & --          & 1.05\e{-1}--8.31\e{0} & 8.44\e{-1} & --    \\
& 2.20\e{3} & 1.07\e{10} & --          & 1.05\e{-1}--7.52\e{0} & 8.41\e{-1} & --    \\
& 4.64\e{4} & 9.99\e{13} & --          & 1.05\e{-1}--7.36\e{0} & 8.37\e{-1} & --    \\
\addlinespace[.55em]
\end{tabular}

\caption{
  \label{tab:results}
The performance and accuracy of our algorithm, compared to naive summation,
for the lattices described by the parameter $m$ and the lattice matrices 
defined in \cref{tab:lattices}.  
}

\end{table}

\begin{figure}
    \centering
    \includegraphics[width=.57\textwidth]{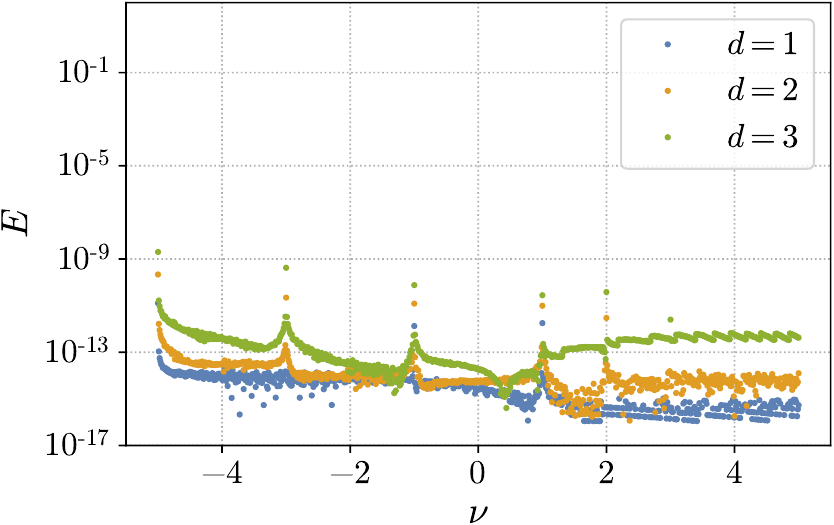}
    \caption{The error plotted for a range of $\nu$, for $d=1,2,3$.}
    \label{fig:nuplot}
\end{figure}

To show the effect of varying $\nu$, we also report the error
$E=\min(E_\text{abs},E_\text{rel})$, evaluated at $\bm x=A
(-\tfrac{1}{2})_{j=1}^d$ and $\bm y=0$, for $\nu$
from $-5+10^{-4}$ to $5+10^{-4}$, with values sampled in increments of
$1/80$.  For $d=1$, we consider the lattice $L$ with $m=10^4$ and lattice
matrix $A_1$; for $d=2$, we consider the lattice $L$ with $m=10^2$ and
lattice matrix $A_2$; and for $d=3$, we consider the lattice $L$ with $m=22$
and lattice matrix $A_3$. The results are shown in \cref{fig:nuplot}.

\begin{figure}
    \centering
    \begin{subfigure}[t]{\textwidth}
      \caption{$d=2$}
      \includegraphics[width=.49\textwidth]{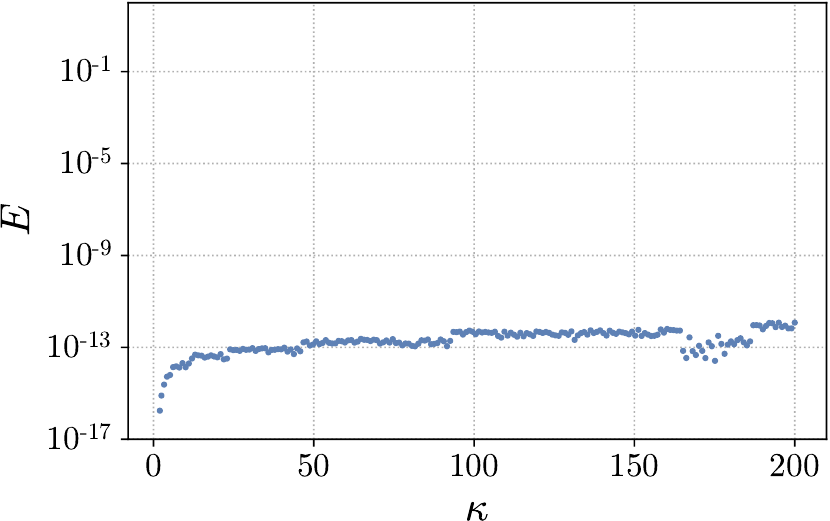}
      \includegraphics[width=.49\textwidth]{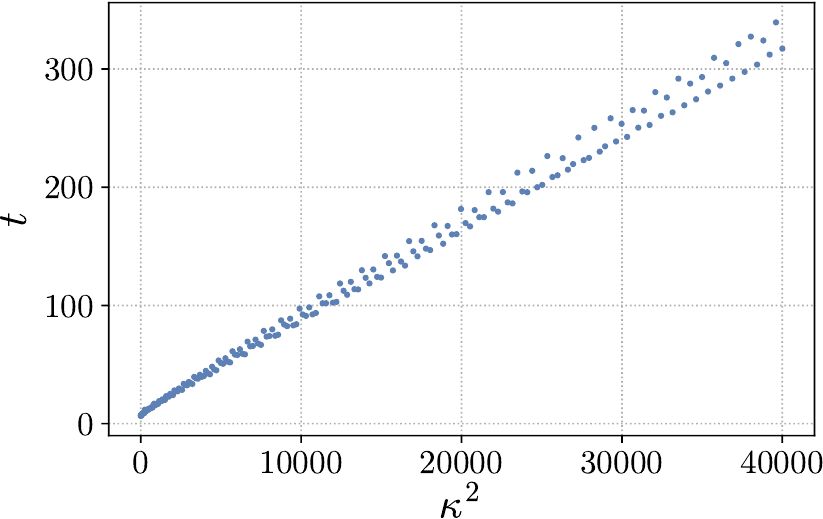}
    \end{subfigure}
    \begin{subfigure}[t]{\textwidth}
      \caption{$d=3$}
      \includegraphics[width=.49\textwidth]{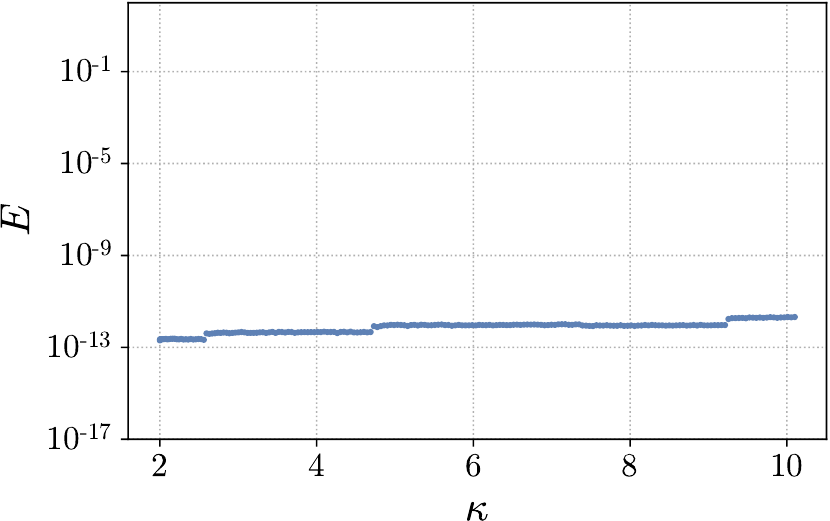}
      \includegraphics[width=.49\textwidth]{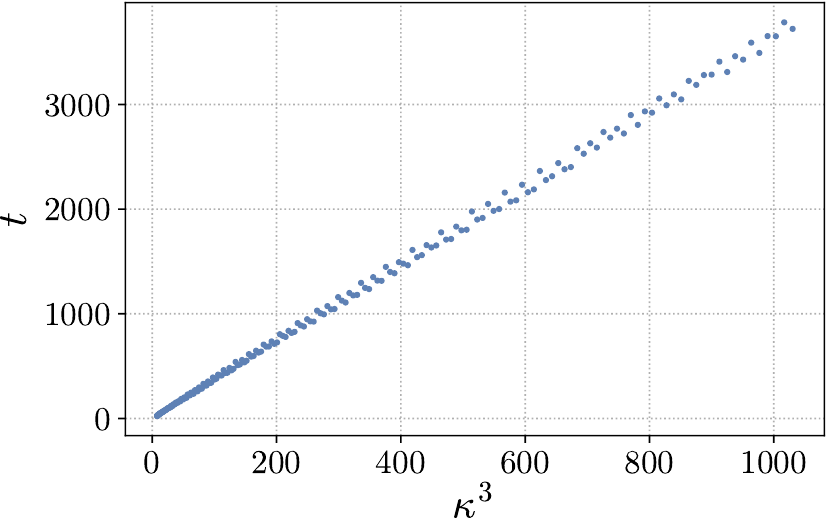}
    \end{subfigure}
    \caption{The errors (left) and runtimes (right) for a range of condition
    numbers $\kappa(A)$.}
    \label{fig:condplot}
\end{figure}

We demonstrate the effect of increasing the condition number of the matrix
$A$, reporting the error $E=\min(E_\text{abs},E_\text{rel})$ evaluated at
$\bm x=A (-\tfrac{1}{2})_{j=1}^d$ and $\bm y=0$, for $\nu=d+0.1$, with the
lattice $L$ with $m=10^2$ in $d=2$ dimensions, and the lattice $L$ with
$m=22$ in $d=3$ dimensions, for the matrices
\begin{align*}
A=\left(\begin{array}{cc}
1 & 0 \\
1 & 1/\mu
\end{array}\right)
\qquad \text{and} \qquad
A=\left(\begin{array}{ccc}
1 & 0 & 0\\
1 & 1/\mu & 0\\
0 & 0 & 1
\end{array}\right),
\end{align*}
since for both of these matrices $\kappa(A) \approx \mu$.  For $d=2$, we let
$\mu$ vary over 201 equispaced values between $1$ and $200$; for $d=3$, we
let $\mu$ vary over 201 equispaced values between $1$ and $10$. We also
report the runtime $t$ in seconds.  The results are shown in
\cref{fig:condplot}.  The cost is seen to grow like $O(\kappa(A)^d)$, as
mentioned in \cref{sec:splitting}.

\section*{Acknowledgements}
The authors gratefully acknowledge the scientific support and HPC resources provided by the Erlangen National High Performance Computing Center (NHR@FAU) of the Friedrich-Alexander-Universit\"at Erlangen-N\"urnberg (FAU) under the NHR project n101af. NHR funding is provided by federal and Bavarian state authorities. NHR@FAU hardware is partially funded by the German Research Foundation (DFG) – 440719683. TK acknowledges funding received from the European Union’s Horizon 2020 research
and innovation programme under the Marie Skłodowska--Curie grant agreement No 899987.
KS's work was supported in part by the NSERC Discovery Grants
RGPIN-2020-06022 and DGECR-2020-00356.
\begin{center}
\includegraphics[width=0.25\linewidth]{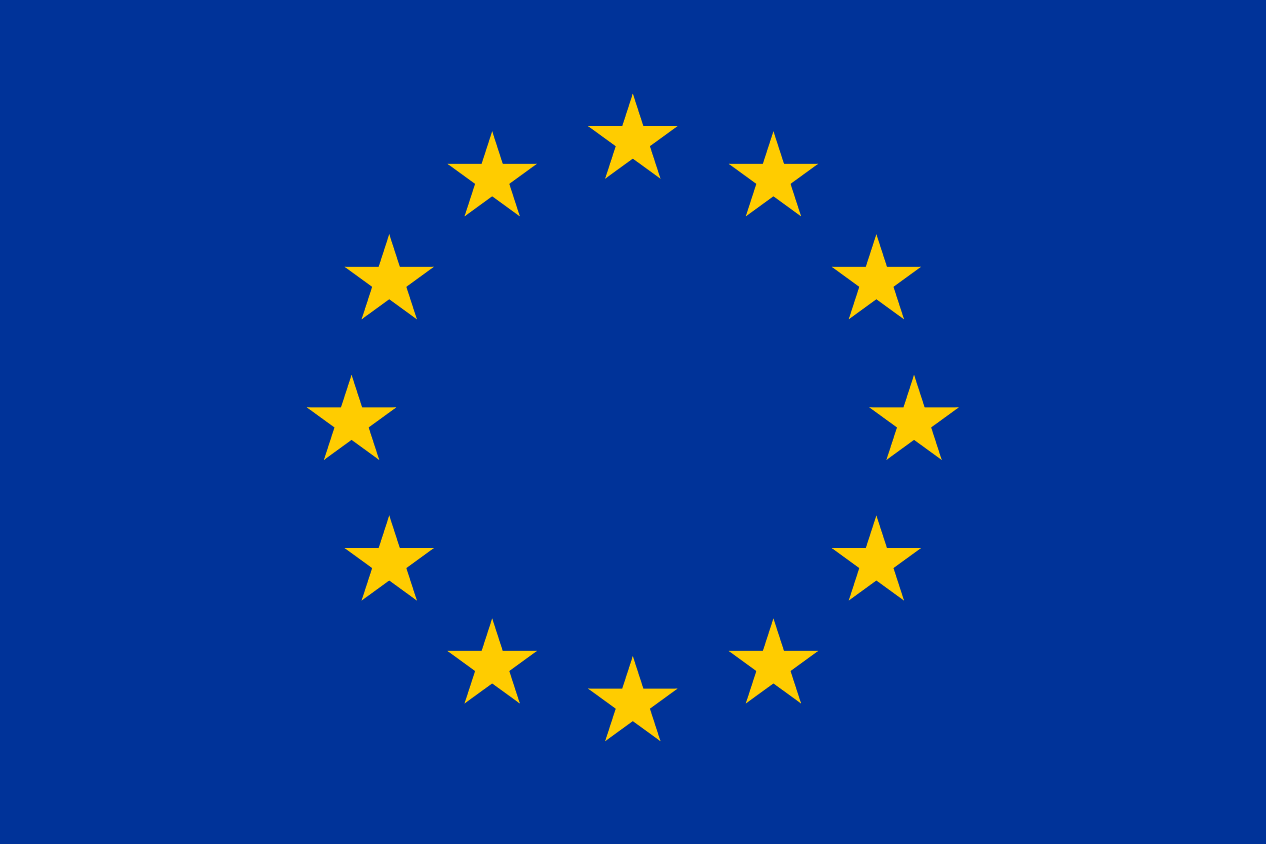}
\end{center}
\appendix 

%%%%%%%%%%%%%%%%%%%%%%%%%%%%%%%%%%%%%%%%%%%%%%%%%%%%%%%%%%%%%%%%%%%%%%%%%%%%%
% Part 4: Appendix
%%%%%%%%%%%%%%%%%%%%%%%%%%%%%%%%%%%%%%%%%%%%%%%%%%%%%%%%%%%%%%%%%%%%%%%%%%%%%

\section{Technical lemmas} \label{sec:technical_lemmas}

Here, we collect results that are needed to formulate the proofs of all theorems
but that are not relevant for the understanding of the main results.

\subsection{Properties of $G_\nu$}

We split the proof of our central theorem into several lemmas.
The first ones are concerned with the properties of $G_\nu$.

\begin{proof}[Proof of \cref{lem:exp-int-sk}]
The topology on $\mathcal S^k(\mathds R^d)$ is generated by the seminorms
\[
q_{\bm \alpha, \bm \beta}(\varphi) = \sup_{\bm z \in \mathds R^d} \big| \bm z^{\bm \alpha} \partial^{\bm \beta} \varphi(\bm z) \big|,
\quad \bm \alpha, \bm \beta \in \Nz^d, \quad |\bm \beta| \leq k.
\]
To show convergence, we split $\mathds R^d$ into a compact set comprising the origin and an unbounded set.
To that end, let $\chi \in \mathcal D(\mathds R^d)$ with $\chi = 1$ on a neighborhood of $\bm 0$.
By the triangle inequality for seminorms,
\[
q_{\bm \alpha, \bm \beta}(G_{-\mu} - \gamma_n)
\leq q_{\bm \alpha, \bm \beta}\big( \chi G_{-\mu} - \chi \gamma_n \big)
+ q_{\bm \alpha, \bm \beta}\big( (1 - \chi) G_{-\mu} - (1 - \chi) \gamma_n \big),
\]
so we can study the bounded and unbounded case separately.
For the latter, note that $(1 - \chi) G_{-\mu}$ and the sequence $(1 - \chi) \gamma_n$, $n \in \Nz$,
are elements of $\mathcal S(\mathds R^d)$, so by \cref{lem:bootstrap-convergence} it requires to show that
\[
\lim_{n \to \infty} (1 - \chi(\bm z)) \gamma_n(\bm z) = (1 - \chi(\bm z)) G_{- \mu}(\bm z), \quad \bm z \in \mathds R^d,
\]
and that $((1 - \chi) \gamma_n)_n$ is bounded.
The pointwise convergence follows immediately by noting that $\gamma_n$ is a Riemann sum approximating $G_{-\mu}$.
To show boundedness in $\mathcal S(\mathds R^d)$, we use \cref{lem:smooth-schwartz-mapping}, applied to
\[
\varphi(\bm z) = \exp(- \pi |\bm z|^2), \quad z \in \mathds R^d
\]
and note that
\[
\gamma_n = \frac{1}{n} \sum_{j = 1}^n (j / n)^{\mu - 1} (1 - \chi) \varphi(\bm \cdot / (j / n)),
\]
so that if $(1 - \chi) \varphi(\bm \cdot / t)$ is bounded by a constant $C > 0$ in a seminorm of $\mathcal S(\mathds R^d)$,
independently of $t > 0$, then $\gamma_n$ is bounded by the same constant in the same seminorm.
$\bm 0$, $\bar B_1(\bm 0)$ and the rest, $\mathds R^d \setminus \bar B_1(\bm 0)$.

Next, we turn to convergence of $\chi \gamma_n$ to $\chi G_{-\mu}$.
By assumption on $\mu$, $G_{-\mu}$ is $k$ times continuously differentiable with
\[
\partial^{\bm \beta} G_{-\mu}(\bm z) = \sum_{|\bm \delta| \leq |\bm \beta|} c_{\bm \delta}
\int_0^1 2 t^{\mu - |\bm \beta| - 1} (\bm z / t)^{\bm \delta} \exp(- \pi |\bm z|^2 / t^2) \, \text d t,
\]
for some constants $c_{\bm \delta}$.
With the same coefficients, it also holds
\[
\partial^{\bm \beta} \gamma_n(\bm z) = \sum_{|\bm \delta| \leq |\bm \beta|} c_{\bm \delta}
\frac{1}{n} \sum_{j = 1}^n 2 (j / n)^{\mu - |\bm \beta| - 1} (n / j \bm z)^{\bm \delta} \exp(- \pi n^2 / j^2 |\bm z|^2).
\]
Thus, we can focus on the convergence of each summand.
To that end, we study functions of the form
$\eta: [0, 1] \times \mathds R^d \to \mathds C$ with
\[
\eta(t, \bm z) = 2 t^{\mu - |\bm \beta| - 1} (\bm z / t)^{\bm \delta} \exp(- \pi |\bm z|^2 / t^2)
\]
for $t > 0$ and $\eta(0, \bm z) = 0$.
The function $\eta$ is continuous as factors including $\bm z$ are always bounded uniformly in $t$, independently of $\bm \delta$,
so that
\[
|\eta(t, \bm z)| \leq C t^{\Re \mu - |\beta| - 1},
\]
for a constant $C > 0$, which tends to $0$ as $t \to 0$, owing to the choice of $\mu$ and $\bm \beta$.
Multiplied with a smooth function of compact support it is uniformly continuous and hence
the Riemann sums $\chi \gamma_n$ converge uniformly in all derivatives up to order $k$ to $\chi G_{-\mu}$.
\end{proof}

\begin{lemma}\label{lem:exp-int-fourier}
Let $\beta > 0$.
The Fourier transformation of the smooth function
\[
\bm z \mapsto \int_0^{\beta} 2 t^{\nu - 1} e^{-\pi \bm z^2 t^2} \, \mathrm d t
\]
for $\nu > d + 1$ is given by $\beta^{-d + \nu} G_{d - \nu}(\bm \cdot / \beta)$.
\end{lemma}

\begin{proof}
Let $h: \mathds R^d \to \mathds C$ denote
\[
h(\bm z) = \int_0^{\beta} 2 t^{\nu - 1} e^{-\pi \bm z^2 t^2} \, \text d t.
\]
To compute the Fourier transform, we discretize the integral
from $0$ to $\beta$ by a Riemann sum, which results in Schwartz functions
\[
h_n(\bm z) = \frac{\beta}{n} \sum_{j = 1}^n 2 t_j^{\nu - 1}
e^{- \pi \bm z^2 t_j^2},
\quad t_j = \frac{\beta}{n} j.
\]
Since the integrand is continuous, we have for each $\bm z \in \mathds R^d$,
\[
\lim_{n \to \infty} h_n(\bm z) = h(\bm z).
\]
Additionally,
\[
|h_n(\bm z)| \leq \frac{\beta}{n} \sum_{j = 1}^n t_j^{\nu - 1},
\quad \bm z \in \mathds R^d.
\]
As the right-hand side converges to the integral of the continuous 
function $t \mapsto 2 t^{\nu - 1}$ over $[0, \beta]$,
it is uniformly bounded in $n \in \Nz$.
Thus,$(h_n)_n$ converges to $h$ in $\mathcal S'(\mathds R^d)$.
In particular, the Fourier transformations of $(h_n)_n$
converge to $\mathcal F h$.
The Fourier transformation is readily computed as
\[
\mathcal Fh_n(\bm k) = 
\frac{\beta}{n} \sum_{j = 1}^n 2 t_j^{\nu - d - 1}
e^{- \pi \bm k^2 / t_j^2}, \quad \bm k \in \mathds R^d.
\]
By assumption, $\nu > d + 1$, so $t \mapsto 2 t^{\nu - d - 1}$
is continuous and bounded on $[0, \beta]$.
Hence, $\mathcal F h_n(\bm k)$ converges to
\[
\int_{0}^\beta t^{\nu - d - 1} e^{-\pi \bm k^2 / t^2} \, \text d t
= \beta^{-d + \nu} G_{d - \nu}(\bm k / \beta),
\]
which proves the asserted form of $\mathcal F h$.
\end{proof}

\subsection{Hadamard representation}
\label{sec:techincal-hadamard-representation}
In the following we collect the necessary lemmas for the proof of \cref{lem:scaled-conv-smooth-continuous}.

\begin{lemma}\label{lem:differentiable-schwartz}
Let $\beta: [0, 1] \times \mathds R^d \to \mathds C$ be a smooth mapping such that for all $t \in [0, 1]$, $\beta(t, \bm \cdot)$ is a Schwartz function.
If the families of partial derivatives $t \mapsto \partial_t^k \beta(t, \bm \cdot)$, $k \in \Nz$,
are uniformly bounded in the Schwartz topology, then $\beta$ is smooth as a mapping
of $[0, 1]$ to $\mathcal S(\mathds R^d)$,
\[
[0, 1] \to \mathcal S(\mathds R^d),~t \mapsto \beta(t, \bm \cdot),
\]
and $\partial_t^k \beta(t, \bm \cdot) \in \mathcal S(\mathds R^d)$ for all $t \in [0, 1]$, $k \in \Nz$.
\end{lemma}

\begin{proof}
We focus on the first derivative. An inductive argument can be used to prove the assertion for an arbitrary derivative order.
First, we write the finite difference as an integral,
\[
\frac{\beta(t + h, \bm z) - \beta(t, \bm z)}{h} = \int_0^1 \partial_t \beta(t + h \tau, \bm z) \, \text d \tau,
\quad \bm z \in \mathds R^d.
\]
By assumption, the integral on the right hand side is uniformly bounded in $t, h$ in all Schwartz seminorms.
Thus, the finite difference is bounded too.
Furthermore, the finite difference converges for fixed $\bm z \in \mathds R^d$ to $\partial_t \beta(t, \bm z)$.
Owing to \cref{lem:bootstrap-convergence}, it converges strongly in $S(\mathds R^d)$
with $\partial_t \beta(t, \bm \cdot) \in \mathcal S(\mathds R^d)$ for all $t \in [0, 1]$.
\end{proof}

\begin{lemma}\label{lem:smooth-schwartz-mapping}
For $\varphi \in \mathcal S(\mathds R^d)$ and $\chi \in \mathcal D(\mathds  R^d)$
with $\chi = 1$ on a neighborhood of~$\bm 0$, the mapping
\[
\beta: [0, 1] \times \mathds R^d  \to \mathds C,~(t, \bm z) \mapsto \begin{cases}
(1 - \chi(\bm z)) \varphi(\bm z / t), & t \neq 0 \\
0, & t = 0
\end{cases}
\]
is smooth as a mapping $[0, 1] \to \mathcal S(\mathds R^d)$ with vanishing derivatives at $t = 0$.
Furthermore, the Schwartz seminorms of $\partial_t^k \beta(t, \bm \cdot)$, $k \in \Nz$,
are uniformly bounded in $t \in [0, 1]$ by seminorms of $\varphi$.
\end{lemma}

\begin{proof}
First, we note that for any $t \in [0, 1]$, $\beta(t, \bm \cdot)$ is a Schwartz function, so the mapping is well-defined.
Moreover, it is a smooth function on $(0, 1] \times \mathds R^d$.
Its partial derivatives are linear combinations of
\[
(t, \bm z) \mapsto \partial^{\bm \alpha}(1 - \chi)(\bm z) t^{-n} \bm z^{\bm \delta} \partial^{\bm \gamma} \varphi(\bm z / t),
\]
for $\bm \alpha, \bm \gamma, \bm \delta \in \Nz^d$ and $n \in \Nz$.
To prove differentiability for $t = 0$, we claim that all derivatives in $t = 0$ vanish
and consider the finite difference of above term for $\bm z \in \mathds R^d$,
\[
\lim_{t \to 0} \frac{\partial^{\bm \alpha}(1 - \chi)(\bm z) t^{-n} \bm z^{\bm \delta} \partial^{\bm \gamma} \varphi(\bm z / t)}{t}
= \lim_{t \to 0} \partial^{\bm \alpha}(1 - \chi)(\bm z) t^{-n - 1} \bm z^{\bm \delta} \partial^{\bm \gamma} \varphi(\bm z / t).
\]
The limit can be rewritten as
\[
\begin{multlined}
\lim_{t \to 0} \partial^{\bm \alpha}(1 - \chi)(\bm z) t^{-n - 1} \bm z^{\bm \delta} \partial^{\bm \gamma} \varphi(\bm z / t) \\
= \lim_{t \to 0} t^{|\bm \delta| + 1} \frac{\partial^{\bm \alpha}(1 - \chi)(\bm z)}{z_1^{n + 2}} (z_1 / t)^{n + 2} (\bm z / t)^{\bm \delta} \partial^{\bm \gamma} \varphi(\bm z / t),
\end{multlined}
\]
where the term involving $\chi$ is zero already on a neighborhood of zero.
In particular, this a smooth and bounded function on $\mathds R^d$, regardless of the singular $1 / z_1^{n + 2}$ at the origin.
Hence, we can bound the term as
\[
\begin{multlined}
\left| t^{|\bm \delta| + 1} \frac{\partial^{\bm \alpha}(1 - \chi)(\bm z)}{z_1^{n + 2}} (z_1 / t)^{n + 2} (\bm z / t)^{\bm \delta} \partial^{\bm \gamma} \varphi(\bm z / t) \right| \\
\leq t^{|\bm \delta| + 1} \sup_{\bm x \in \mathds R^d} \left| \frac{\partial^{\bm \alpha}(1 - \chi)(\bm x)}{x_1^{n + 2}} \right|
\sup_{\bm x \in \mathds R^d} |x_1^{n + 2} \bm x^{\bm \delta} \partial^{\bm \gamma} \varphi(\bm x)|,
\end{multlined}
\]
so that the limit $t \to 0$ is zero.
Therefore, $\beta$ is smooth on $[0, 1] \times \mathds R^d$ and all partial derivatives vanish at $t = 0$.
Moreover, above estimates show that Schwartz seminorms with respect to $\bm z$ of all partial derivatives of $\beta$
are uniformly bounded for $t \in [0, 1]$ in terms of Schwartz seminorms of $\varphi$.
The assertion now follows from \cref{lem:smooth-schwartz-mapping}.
\end{proof}

\begin{proof}[Proof of \cref{lem:scaled-conv-smooth-continuous}]
For $\bm y_0 \in U$ there exits an open neighborhood $\Omega \subseteq U$ of $\bm y_0$ and $\varepsilon > 0$.
such that $\bm y + B_{3 \varepsilon}(\bm 0) \subseteq U$ for all $\bm y \in \Omega$.
Hence, for $\chi \in \mathcal D(\mathds R^d)$ even with $\supp \chi \subseteq B_{2 \varepsilon}(\bm 0)$,
$0 \leq \chi \leq 1$,
and $\chi = 1$ on $B_{\varepsilon}(\bm 0)$,
the distribution $\tau_{\bm y}\chi u$ is a smooth function with compact support for every $\bm y \in \Omega$.
If we split $T \varphi$ into two parts,
\[
T \varphi(t, \bm y) = \big( \big[ \tau_{\bm y} \chi u \big] \ast \varphi(\bm \cdot / t) \big)(\bm y)
+ \big( \big[ (1 - \tau_{\bm y} \chi) u \big] \ast \varphi(\bm \cdot / t) \big)(\bm y),
\]
then for first term, denoted by $T_1 \varphi$, we have
\[
T_1 \varphi(t, \bm y) = t^d \int_{\mathds R^d} \chi(t \bm z) u(\bm y - t \bm z) \varphi(\bm z) \, \text d \bm z,
\quad t \geq 0,~\bm y \in \Omega.
\]
Owing to the choice of the cut-off function $\chi$,
\[
|\chi(t \bm z) u(\bm y - t \bm z)| \leq \sup_{\bm z \in \Omega + B_{2 \varepsilon}(\bm 0)} |u(\bm z)|.
\]
Similar bounds also hold for arbitrary derivatives with respect to $t$ and $\bm y$ of the integrand.
By the dominated convergence theorem, $T_1 \varphi$ is a smooth function of both arguments.
Furthermore, $T_1 \varphi$ is bounded by
\[
|T_1 \varphi(t, \bm y)| \leq \sup_{\bm z \in \Omega + B_{2 \varepsilon}(\bm 0)} |u(\bm z)|
\sup_{\bm z \in \mathds R^d} |(1 + |\bm z|^2)^d \varphi(\bm z)|
\int_{\mathds R^d} (1 + |\bm z|^2)^{-d} \, \text d \bm z,
\]
and similarly for its derivatives.
This proves that $T_1$ is a continuous linear mapping of $\mathcal S(\mathds R^d)$
to $C^\infty([0, 1] \times \Omega)$.
Since $\bm y$ was arbitrary, the smoothness around any point of $U$ implies that $T_1$
maps onto smooth functions on the whole complement of the singular support,
that is
\[
T_1 : \mathcal S(\mathds R^d) \to C^\infty([0, 1] \times U)
\]
is well-defined, linear and continuous.

We now focus on the second term, $T_2 \varphi$.
Since $\chi$ is even, the convolution in $T_2 \varphi$ can be rewritten as
\[
T_2 \varphi(t, \bm y)
= u \ast \big[ (1 - \chi) \varphi(\bm \cdot / t) \big](\bm y)
= \big[ u \ast \beta(t, \bm \cdot) \big](\bm y)
\]
with $\beta(t, \bm z) = (1 - \chi(\bm z)) \varphi(\bm z / t)$.
For any $t \in [0, 1]$, $\beta(t, \bm \cdot)$ is a Schwartz function,
so the convolution is a smooth function of $\bm y$.
Furthermore, for fixed $\bm y \in \mathds R^d$,
the convolution is a smooth function of $t \in [0, 1]$ by \cref{lem:smooth-schwartz-mapping}.
Thus, $T_2 \varphi \in C^\infty([0, 1] \times \mathds R^d)$.
To show continuity of the mapping, we observe that any partial derivative of degree $k$ in $t$
and of arbitrary order in $\bm y$ is bounded compactly in $\bm y$ by seminorms
of $\partial_t^k \beta(t, \bm \cdot)$, that, by \cref{lem:smooth-schwartz-mapping},
are uniformly bounded in $t$ by seminorms of $\varphi$.
Thus, $T_2 \varphi$ is bounded and consequently continuous.

We conclude with a proof of the addendum for an even test function $\varphi$.
For $t \in [-1, 1]$, $t \neq 0$, we have
\[
\varphi(\bm z / t) = \varphi(\bm z / |t|), \quad \bm z \in \mathds R^d.
\]
In particular, we can smoothly extend $T_2 \varphi$ from $(0, 1]$ to $[-1, 1] \setminus \{ 0 \}$ via
\[
T_2 \varphi(t, \bm y) = T_2 \varphi(|t|, \bm y).
\]
For continuation into $t = 0$ we note that all derivatives of $T_2 \varphi$ decay super-algebraically for $t \to 0$
as a result of the proof for \cref{lem:smooth-schwartz-mapping}.
Hence the function
\[
g(t, \bm y) = |t|^{-d} T_2 \varphi(|t|, \bm y), \quad t \in [-1, 1] \setminus \{ 0 \}, \quad \bm y \in \mathds R^d
\]
has a smooth extension for $t = 0$ with $g^{(k)}(0) = 0$ for all $k \in \Nz$.
This shows
\[
g(t, \bm y) = \eta_2(t^2, \bm y), \quad t \in [-1, 1],~\bm y \in \mathds R^d,
\]
for $\eta \in C^\infty([0, 1] \in \mathds R^d)$, so that
\[
T_2 \varphi(t, \bm y) = t^d \eta_2(t^2, \bm y), \quad t \in [0, 1].
\]
Because $\varphi$ is even, we can write
\begin{align*}
\int_{\mathds R^d} \chi(t (- \bm z)) u(\bm y - t (- \bm z)) \varphi(\bm z) \, \text d \bm z
&= \int_{\mathds R^d} \chi(t \bm z) u(\bm y - t \bm z) \varphi(\bm z) \, \text d \bm z
\end{align*}
for $t \in [0, 1]$, $y \in \Omega$,
which shows that
\[
T_1(t, \bm y) = t^d v(t, \bm y), \quad t \in [0, 1],~\bm y \in U,
\]
for an even, smooth function $v$.
Hence, there is $\eta_1 \in C^\infty([0, 1] \times U)$ such that
\[
T_1 \varphi(t, \bm y) = t^d \eta_1(t^2, \bm y).
\]
The combination of both results proves the desired assertion.
\end{proof}

\printbibliography

\end{document}